\documentclass[10pt,a4paper]{amsart}
\usepackage{amsthm,amsmath,amssymb}

\usepackage{bm}
\usepackage{subfig,tikz}
\usepackage{wrapfig}
\usepackage{color}
\usepackage{xspace}
\usetikzlibrary{shapes.geometric}

\usepackage[ margin=1.5in]{geometry}
     \usepackage{amssymb}
     \usepackage{url}
	\usepackage{graphicx}
           \graphicspath{{c:/users/ali/pictures/}}

      \theoremstyle{plain}
      \newtheorem{theorem}{Theorem}[section]
      \newtheorem{lemma}[theorem]{Lemma}
      \newtheorem{corollary}[theorem]{Corollary}
      \theoremstyle{Conjecture}
       \newtheorem{Conjecture}[theorem]{Conjecture}
      \theoremstyle{definition}
      \newtheorem{definition}[theorem]{Definition}
      
      \theoremstyle{remark}
      \newtheorem{remark}[theorem]{Remark}

      \makeatletter
      \def\@setcopyright{}
      \def\serieslogo@{}
      \makeatother
   
\usepackage[english]{babel}

\DeclareMathOperator{\spn}{span}

\begin{document}

   \author {{Ali Feizmohammadi}}
   \address{University College London}
   \email{a.feizmohammadi@ucl.ac.uk}

   % second author

 %  \author{Stephen G.~Simpson}

   % the address where the research was carried out
   %\address{University of G\"ottingen, G\"ottingen, Germany}

   % current address, usually not needed because it is the same as the
   % regular address
   %\curraddr{Department of Mathematics, University of Toronto,
     %University Park, State College PA 16802}

   %\email{simpson@math.psu.edu}
   
   % title

%\tableofcontents
%\newpage

   \title[DN map and Locally Euclidean Geometries]{Uniqueness of a Potential from Boundary Data in Locally Conformally Transversally Anisotropic Geometries}

   % Note that the short title for running heads goes in square
   % brackets.  This is optional.  The long title goes in curly
   % braces.  In the long title, line breaks are indicated by \\.

   % abstract (optional)
   \begin{abstract}
      Let $(\Omega^3,g)$ be a compact smooth Riemannian manifold with smooth boundary and suppose that $U$ is a an open set in $\Omega$ such that $g|_U$ is the Euclidean metric. Let $\Gamma= \overline{U} \cap \partial \Omega$ be connected and suppose that $U$ is the convex hull of $\Gamma$. We will study the uniqueness of an unknown potential for the Schr\"{o}dinger operator $ -\triangle_g +  q  $ from the associated Dirichlet to Neumann map, $\Lambda_q$. We will prove that if the potential $q$ is a priori explicitly known in $U^c$, then one can uniquely reconstruct $q$ over the convex hull of $\Gamma$ from $\Lambda_q$. We will also outline a reconstruction algorithm. More generally we will discuss the cases where $\Gamma$ is not connected or $g|_{U}$ is conformally transversally anisotropic and derive the analogous result.
   \end{abstract}

   % AMS subject classifications (used in AMS journals)
   %\subjclass{Primary 00A30; Secondary 00A22, 03E20}

   % AMS keywords (used in AMS journals)
   %\keywords{infinite, seminar}

   % acknowledge support, etc
  % \thanks{This research was partially supported by NSF grant
     %DOA-123456789.}
   %\thanks{We would like to thank our colleagues for their helpful
    % criticism.}

   % dedication
   %\dedicatory{Dedicated to Professor Donald Knuth on the occasion
     %of his $100$th birthday}

   % today's date, or fill in whatever date you prefer
  % \date{\today}

% This ends the top matter information.
% We can now tell LaTeX to display the top matter.

   \maketitle

% Having displayed the top matter, we now proceed to the body of the
% article.

% The body of the article is divided into sections.
% Each section begins with a \section command.

\tableofcontents

\section{Introduction}
 %Suppose that $(\Omega^3,g)$ is a compact smooth three dimensional Riemannian manifold. Let us define the Dirichlet to Neumann map $\Lambda_{q}: H^{\frac{1}{2}}(\partial \Omega) \to  H^{-\frac{1}{2}}(\partial \Omega)$ for Schr\"{o}dinger operator as follows:
%\[ (\Lambda_{q} f, h) = \int_{\Omega} \langle du,dv\rangle_g dV_g + \int_{\Omega} quv dV_g \]
%where $ v \in H^1(\Omega)$ satisfies $ v|_{\partial\Omega} = h$ and $ u \in H^1(\Omega)$ satisfies $ (-\triangle_g+q) u=0$ with $ u|_{\partial\Omega}=f$. Let us impose the condition that $0$ is not a Dirichlet eigenvalue for $-\triangle_g + q$.  $\Lambda_q$ is a self adjoint bounded linear Elliptic Pseudo Differential Operator [4]. For smoother Dirichlet data $f \in H^{\frac{3}{2}}(\partial \Omega)$ we have $ \Lambda_q f = \partial_{\nu} u |_{\partial\Omega}$ where $\nu$ is the outward normal unit vector field on the boundary of the domain.\\
 In 1980, Calder\'{o}n \cite{C} proposed the following question: Can the conductivity of an unknown medium be determined from voltage and current measurments on the boundary? Since then this question has been of considerable interest in the area of inverse problems and there is a rich literature of results.
Let us assume that $\Omega$ is a domain in $\mathbb{R}^n$ with $C^{\infty}$ boundary. Under the assumption of no sources or sinks in the domain we can represent the current flow in a body through the following elliptic partial differential equation:
 \[
     \left\{\begin{array}{lr}
       \nabla \cdot ( \gamma \nabla u)=0 & \text{for }  x \in \Omega \\
       u=f & \text{for} x \in \partial \Omega\\
        \end{array}\right\}
  \]
\noindent Here $\gamma$ is a positive definite matrix representing the conductivity tensor for the medium and $f$ is a given voltage on the boundary. The normal component of the current flux at the boundary is given by the expression:
\[ \Lambda_{\gamma} f := \gamma \nabla u \cdot \nu |_{\partial \Omega}\]
where $\nu$ represents the outward pointing unit normal vector on $\partial \Omega$. 
The problem of Calder\'{o}n asks whether the knowledge of $\Lambda_\gamma$ uniquely determines the conductivity tensor $\gamma$. While the initial formulation of the problem assumed a scalar (isotropic) conductivity, much current research considers the case of anisotropic conductivities. This is motivated by applications to medical imaging: muscle or heart tissues have conductivity that depends not only on the location, but also on the direction (for instance along or across the muscle fiber). Since the conductivity equation above corresponds to a Laplace Beltrami equation in a Riemannian manifold, to bring to bear essential tools from differential geometry, it is useful to reformulate the question as follows:\\
\noindent Let $(M,g)$ be a compact smooth Riemannian manifold with smooth boundary and consider harmonic functions $u$ with prescribed Dirichlet data on $\partial M$:

\[  
   \left\{\begin{array}{lr}
       \triangle_g u=0 & \text{for }  x \in M \\
       u=f & \text{for} x \in \partial M\\
        \end{array}\right\}
  \] 

\noindent For $f \in H^{\frac{1}{2}}(\partial M)$ the above partial differential equation is known to have a unique solution $u \in H^1(M)$. Inspired by the physical intrepretation of the problem we define the Dirichlet to Neumann map $\Lambda_g$ as a bilinear functional as follows:
\[ (\Lambda_g f, h) = \int_{M} \langle du,dv\rangle_g dV_g\]
where $ v \in H^1(M)$ satisfies $ v|_{\partial M} = h$ and $ u \in H^1(M)$ satisfies $ -\triangle_g u=0$ with $ u|_{\partial M}=f$.
In fact we note that for smoother data $ f \in H^{\frac{3}{2}}(\partial M)$:
\[ \Lambda_g f = \partial_{\nu} u|_{\partial M}. \]

\noindent The Calder\'{o}n problem can then be reformulated as follows: Does the knowledge of the Dirichlet to Neumann map $\Lambda_g$ uniquely determine the unknown metric $g$?There is an immediate obstruction (first noted by Luc Tartar \cite{KV}) to uniqueness due to diffeomorphisms that fix the boundary as can be seen from the following Lemma \cite{DKSU}.
\begin{lemma}
If $F: M \to M$ is a diffeomorphism such that $F|_{\partial M}=I$ then:
\[ \Lambda_{F^{*}g}=\Lambda_g\]
Here $F^{*}g$ denotes the pullback of the metric $g$.
\end{lemma}

\noindent We can now state the Calder\'{o}n conjecture in Riemannian geometries as follows:

\begin{Conjecture}
\label{Calderon}
Let $(M,g)$ denote an n-dimensional Riemannian manifold with smooth boundary with $n \geq 3$. Then $\Lambda_g$ uniquely determines $g$ up to diffeomorphisms that fix the boundary.
\end{Conjecture}

%\begin{remark}
%The Calder\'{o}n conjecture can also be posed for two dimensional Riemannian manifolds. In this case there is an additional obstruction to uniqueness as Laplacian is conformally invariant in two dimensions. The uniqueness in dimension $2$ was proved by Adrian Nachman in [16 ].
%\end{remark}

\noindent  In \cite{LU} it is showed that $\Lambda_g$ is a self-adjoint elliptic pseudo-differential operator of order $1$ and that the symbol of $\Lambda_g$ determines the metric and its jet at the boundary. This yields the solution of the Calder\'{o}n problem in the case where the metric is real analytic with some mild topological assumptions on the manifold.The general case of smooth manifolds is a major open problem.\\ %Substantial progress has been achieved by proving identifiability of metrics in a conformal class that admit a Limiting Carleman Weights [5]. Moreover, it has been proved that this is essentially the widest class one can treat with existing approaches. 
There is a simpler version of Conjecture ~\ref{Calderon} that is concerned with determining the conformal factor when the conformal class of the manifold is known.
%In this thesis, we also work on the problem of determining the conformal factor when the conformal class of the manifold is known. We seek to find techniques that allow us to go beyond the ones obtained assuming existence of Limiting Carleman Weights. 
The standard starting point when one wishes to determine the conformal factor is the observation that (see e.g. \cite{S}):

\[ \triangle_{cg} u = c^{-\frac{n+2}{4}}(\triangle_g + q_c) (c^{\frac{n-2}{4}}) \]

\noindent where $q_c = c^{\frac{n-2}{4}}\triangle_{cg}(c^{-\frac{n-2}{4}})$. It can be shown that if $c|_{\partial M} =1$ and $\partial_\nu c|_{\partial M}=0$ we have:
\[ \Lambda_{cg} = \Lambda_{g,-q_c}\]
where $\Lambda_{g,q}$ in general denotes the Dirichlet to Neumann map (DN map) for the Schr\"{o}dinger equation:
\[
  \left\{\begin{array}{lr}
       (-\triangle_g + q) u=0 & \text{for }  x \in M \\
       u=f & \text{for} x \in \partial M\\
        \end{array}\right\}
  \] 
We will assume that $0$ is not a Dirichlet eigenvalue for this equation (the assumption can easily be removed by using all Cauchy data pairs). $\Lambda_{g,q}: H^{\frac{1}{2}}(\partial M) \to H^{-\frac{1}{2}}(\partial M)$ can be defined weakly through the bilinear pairing:
\[ (\Lambda_{g,q} f, h) = \int_{M} \langle du,dv\rangle_g dV_g + \int_{M} quv dV_g \]
where $ v \in H^1(M)$ satisfies $ v|_{\partial M} = h$ and $ u \in H^1(M)$ satisfies $ (-\triangle_g +q) u=0$ with $ u|_{\partial M}=f$.
As before we note that for smoother Dirichlet data $ f \in H^{\frac{3}{2}}(\partial M)$ we have that:
\[ \Lambda_{g,q} f = \partial_{\nu} u|_{\partial M}. \]

\noindent Thus one can observe that the Calder\'{o}n problem in the conformal setting can be posed as follows:

\begin{Conjecture}
\label{conformal}
Let $(M,g)$ be a smooth compact Riemannian manifold with smooth boundary and let $q \in L^{\infty}(M)$ be an unknown bounded function. Then the knowledge of $\Lambda_{g,q}$ will uniquely determine $q$.
\end{Conjecture}

\begin{remark}
\noindent Henceforth, for the sake of brevity, we will let $\Lambda_q$ to denote $\Lambda_{g,q}$. Here the reader should note that the geometry $g$ is a priori known and $q$ is unknown.
\end{remark}

\noindent For $g$ Euclidean and $n \geq 3$ this conjecture was solved by John Sylvester and Gunther Uhlmann in 1987 \cite{SU}.  In two dimensions the problem was first solved by Adrian Nachman \cite{NII} for potentials that arise from positive functions $c$ as above. A different proof was given by Bukhgeim in 2008 that works for general potentials \cite{B}. For general smooth Riemannian manifolds this is still a difficult open question. The most general result in this direction is the uniqueness of the potential for geometries that are conformally transversally anisotropic (CTA). These are geometries where the manifold has a product structure $M=\mathbb{R} \times M_0$ with $(M_0,g_0)$ being called the transversal manifold and the metric takes the form:
\[ g = c(t,x) (dt^2 + g_0(x)).\]
\begin{definition}
Let $(M_0,g_0)$ denote a smooth Riemannian manifold with boundary. We say $(M_0,g_0)$ is simple if $\partial M_0$ is strictly convex and any two points in $M_0$ can be connected with a unique length minimizing geodesic.
\end{definition}

\bigskip

\noindent D. Dos Santos Ferreira, C. E. Kenig, M. Salo and G. Uhlmann \cite{DKSU} proved in 2009 that the potential $q$ can be uniquely determined from the Dirichlet to Neumann map in CTA gemoetries under the additional assumption that the transversal manifold $(M_0,g_0)$ is simple. In 2013, their result was strengthend by assuming weaker restrictions on the transversal manifold, namely that the geodesic ray transform on $(M_0,g_0)$ is injective \cite{DKLS}. 
\bigskip

\begin{definition}
Let $U_0 \subset M_0$ and $f \in C^1(M_0)$. We say the geodesic ray transform on $M_0$ is locally injective with respect to $U_0$, if the following holds:
$$\int_{\gamma} f  =0$$
for all geodesics $\gamma$ with end points on $\partial M_0 \cap \overline{U_0}$ implies that $$ f|_{U_0}\equiv 0.$$

\end{definition}

\section{Statement of Results}

\noindent This paper is concerned with the proof of the following theorems. These answer, at least in part, a conjecture by Lauri Oksanen \cite{link}.

\begin{theorem}
\label{euclid}
Let $(\Omega^3,g)$ denote a compact smooth Riemannian manifold with smooth boundary. Let $U \subset \Omega$ be an open subset such that $\Gamma=\overline{U} \cap \partial \Omega$ is non-empty, connected and strictly convex. Suppose $\Gamma$ can be covered with a coordinate chart in which $g|_U$ is the Euclidean metric and that $U$ is the convex hull of $\Gamma$.   Suppose $q$ is a smooth function and suppose $q-q_*$ is compactly supported in $U$ where $q_*$ is a globally known smooth function. Then the knowledge of $\Lambda_q$ will uniquely determine $q$.\\ 
\end{theorem}

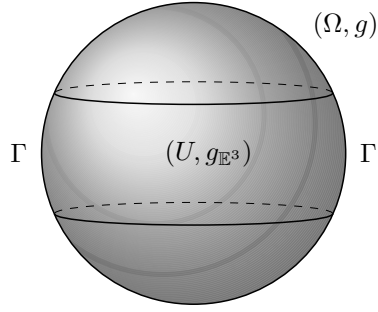
\bigskip

\begin{figure}[h]
           \centering
	\begin{tikzpicture}
	\shade[ball color = gray!40, opacity = 0.4] (0,0) circle (2cm);
	\draw[line width = 0.2mm](0,0) circle (2cm);
	\draw[line width = 0.2mm] (-1.83,0.8) arc (180:360:1.83 and 0.15);
	\draw[line width = 0.1mm, dashed] (1.83,0.8) arc (0:180:1.83 and 0.15);
	\node[scale =1] at (0.2,0) {$(U,g_{\mathbb{E}^3})$};
	\node[scale =1] at (2,1.7) {$(\Omega,g)$};
	\node[scale =1] at (2.3,0) {$\Gamma$};
	\node[scale =1] at (-2.3,0) {$\Gamma$};
	\draw[line width = 0.2mm] (-1.83,-0.8) arc (180:360:1.83 and 0.15);
	\draw[line width = 0.1mm, dashed] (1.83,-0.8) arc (0:180:1.83 and 0.15);
	\end{tikzpicture}
\end{figure}

\noindent If $\Gamma$ is not connected, we have the following result:

\bigskip

\begin{theorem}
\label{support}
Let $(\Omega^3,g)$ denote a compact smooth Riemannian manifold with smooth boundary. Let $U \subset \Omega$ be an open subset such that $\Gamma=\overline{U} \cap \partial \Omega$ is non-empty, and strictly convex. Let $\Gamma=\cup_{i=1}^{l} \Gamma^i$ where $\Gamma^i$ denotes the connected components of $\Gamma$. Suppose $q$ is a smooth function and suppose $q-q_*$ is compactly supported in $U$ where $q_*$ is a globally known smooth function. Furthermore suppose $U$ can be covered with a coordinate chart in which $g|_U$ is the Euclidean metric. Let $U^i$ denote the convex hull of $\Gamma_i$ and let us assume that $U^i \cap U^j = \emptyset \hspace{2mm}\forall i,j$ and $U = \cup_{i=1}^l U^i$. Then the knowledge of $\Lambda_q$ will uniquely determine $q$.\\ 
\end{theorem}

\bigskip

\noindent More generally, we have the following result which does not require $g$ to be Euclidean in $U$:

\bigskip

\begin{theorem}
\label{cta0}
Let $(\Omega^3,g)$ denote a compact smooth Riemannian manifold with smooth boundary with $\Omega=I \times \Omega_0$. Let $U_0$ be an open subset such that $\Gamma_0=\overline{U_0} \cap \partial \Omega_0$ is non-empty and that $U_0$ is the convex hull of $\Gamma_0$. Suppose that $(\partial I \times U_0) \cup (I \times \partial U_0)$ is connected, and let $U= I \times U_0$. Suppose $q$ is a smooth function that is explicitly known in $U^c$, and that $U$ can be covered with a coordinate chart in which $(U, g|_U)$ is conformally transversally anisotropic. Then the knowledge of $\Lambda_q$ will uniquely determine $q$ provided that $U_0$ is simple and the geodesic transform on $\Omega_0$ is locally injective with respect to $U_0$.\\ 
\end{theorem}

\bigskip

%\noindent Let us make a few remarks before we begin the proofs. Note that one may argue that since $q$ is supported in $U$ and since the metric $g$ is known everywhere in $\Omega$ we might be able to determine the Dirichlet to Neumann map for Schr\"{o}dinger equation in the smaller domain $(U,e)$ from the DN map in $(\Omega,g)$. This is however not possible as the set of solutions to $ (-\triangle_g+q)u=0$ in $(\Omega,g)$ is not the same as the set of solutions in $(U,g_{\mathbb{E}^3})$. Therefore it is not a trivial matter to determine the DN map for  $(-\triangle_g+q)u=0$ in $(U,g_{\mathbb{E}^3})$. It is possible though to use a density argument and a quantitative version of Runge approximation as discussed in \cite{R} can be used to conclude uniqueness of potential $q$ in $U$. The proof will not be constructive however and an altogether different approach is probably needed to give a reconstruction algorithm. Motivated by this, we approach the question quite differently by using a type of unique continuation argument. Uniqueness of potential is proved and we sketch out a reconstruction algorithm as well. The key in our reconstruction algortihm is to construct a symmetric Fadeev type green function [10] for the Laplacian operator and use a strong unique continuation argument. [22] 
%\begin{figure}
%\centering
%\includegraphics[width=6cm, height=6cm]{Figure}
%\end{figure}

\section{Acknowledgements}
This work is part of the PHD thesis research of the author at the University of Toronto. I would like to thank Professors Adrian Nachman and Spyros Alexakis for all their support, advice, and all the time that they spent on this work. I am indebted to them for helpful discussions and ideas.

\section{Outline of the  Paper}
Sections 4,5, 6 and 7 are concerned with proving Theorem ~\ref{euclid}. In section $4$ we will prove a Carleman type estimate through Lemma ~\ref{est} which is a key ingredient of the proof. Section $5$ will use this estimate to construct solutions to the Schr\"{o}dinger equation concentrating on 2-planes. These solutions are explicitly used in section $6$ to derive the uniqueness of the potential. Section $7$ contains the reconstruction algorithm. Section 8 delves into the proofs for Theorem ~\ref{support} and Theorem ~\ref{cta0}. 
\bigskip

\noindent Let us make a few remarks before we begin the proofs. Note that one may argue that since $q-q_*$ is supported in $U$ and since the metric $g$ is known everywhere in $\Omega$ we might be able to determine the Dirichlet to Neumann map (DN map) for the Schr\"{o}dinger equation in the smaller domain $U$ from the DN map in $(\Omega,g)$. This is however not immediate as the set of solutions to $ (-\triangle_g+q)u=0$ in $(\Omega,g)$ is not the same as the set of solutions in $(U,g_{\mathbb{E}^3})$. It is possible though to use a density argument and a quantitative version of Runge approximation as discussed in \cite{R} to conclude uniqueness of potential $q$ in $U$. The proof will not be constructive however and an altogether different approach is probably needed to give a reconstruction algorithm. Motivated by this, we approach the question quite differently. Uniqueness of the potential function is proved and we sketch out a reconstruction algorithm as well. The key in our reconstruction algorithm is to construct a symmetric Fadeev type green function (as in \cite{F}) for the Laplacian operator. We will subsequently use a strong unique continuation result to complete the proof  (see e.g. \cite{U}). 
%\begin{figure}

\section{Carleman Estimates}

\theoremstyle{definition}
\begin{definition}{}
A smooth function $\phi$ is called a Carleman weight with respect to $(\Omega_1,g)$  if there exists $h_0>0$ such that the following estimate holds:
\[ \| e^{\frac{\phi}{h}} \triangle_g (e^{-\frac{\phi}{h}} u) \|_{L^2(\Omega_1)} \ge \frac{C}{h} \|u\|_{L^2(\Omega_1)}+ C  \|Du\|_{L^2(\Omega_1)} \]
$\forall 0< h < h_0$ and $ u \in C^{\infty}_{c} (\Omega_1)$.\\ 
\end{definition}

\noindent Let us extend the manifold $\Omega$ to a slightly larger manifold $\Omega_1$. We extend $q$ to all of $\Omega_1$ by setting it equal to zero outside $\Omega$ and extend $g$ smoothly to $\Omega_1$ such that $ g|_{U_1 }$ is Euclidean. Here $U_1$ denotes the extension of $U$ to the larger manifold $\Omega_1$. Note that $U_1$ has a foliation by a family of planes $\mathbb{A}=\{\Pi_t\}_{t \in I}.$ We start by taking a fixed plane $\Pi \in \mathbb{A}$. A local coordinate system $(x_1,x_2,x_3)$ can be constructed in $U_1$ such that $\Pi=\{x_3=0\}$ with $(x_1,x_2)$ denoting the usual cartesian coordinate system on the plane $\Pi$ and $\partial_3$ denoting the normal flow to this plane. We can assume that support of $q$ lies in the compact set $V \subset\subset \{-t_1\leq x_3 \leq t_2\}$ with $t_1,t_2>0$. In this framework $U_1= \cup_{c=-t_1-\delta_1}^{c=t_2+\delta_2} \{x_3=c\}$ with $\delta_i >0$ for $i \in \{1,2\}$.\\
\begin{definition}
Let us define two smooth functions $\omega:\Omega_1 \to \mathbb{R}$ and $\tilde{\omega}:\Omega_1 \to \mathbb{R}$ as follows:
\begin{itemize}

\item{Let $\omega:\Omega_1 \to \mathbb{R}$ be any smooth function such that $d\omega \neq 0$ everywhere in $\Omega_1$ and $\omega(x) \equiv x_3$  for $ x \in U_1$.}
\item{Let $\tilde{\omega}:\Omega_1 \to \mathbb{R}$ be any smooth function such that $\tilde{\omega}(x) \equiv x_2$  for $ x \in U_1$.}
\end{itemize}
\end{definition} 
\begin{definition}
Let us define two globally defined $C^{k-1}(\overline{\Omega_1})$ functions $\chi_0 : \Omega_1 \to \mathbb{R}$ and $F_{\lambda}: \mathbb{R} \to \mathbb{R}$  as follows:
 \[
    \chi_0(x) = \left\{\begin{array}{lr}
        1, & \text{for }  -t_1<x_3<t_2\\
        (1-(\frac{x_3-t_2}{\delta_2})^{8k})^k, & \text{for } t_2 \leq x_3 \leq t_2+\delta_2\\
    (1-(\frac{x_3+t_1}{\delta_1})^{8k})^k, & \text{for } -t_1-\delta_1 \leq x_3 \leq -t_1\\
         0  & \text{otherwise }
        \end{array}\right\}
  \]

 \[
    F_{\lambda} (x) = \left\{\begin{array}{lr}
        0, & \text{for }   -t_1<x<t_2\\
        e^{\lambda (\frac{x-t_2}{\delta_2})^2} (\frac{x-t_2}{\delta_2})^{2k}, & \text{for } t_2 \leq x  \\
       e^{\lambda (\frac{x+t_1}{\delta_1})^2} (\frac{x+t_1}{\delta_1})^{2k}  , & \text{for }  x \leq -t_1\\
        \end{array}\right\}\\
\\
  \]
\end{definition}

\noindent We will utilize the functions defined above to construct an appropriate global Carleman weight in the entire domain $(\Omega_1,g)$. The key idea here is to start with a local limiting Carleman weight over the Euclidean neighborhood (for example $\phi(x)=x_1$) and extend it smoothly to the entire manifold in a way that it will satisfy the H\"{o}rmander hypoellipticity condition in one direction. We believe this is the first time that the idea of extending a Limiting Carleman Weight is being implemented to solve an inverse problem.
\\

\begin{lemma}
\label{est}

Let $ \tilde{\phi}_0(x_1,x_2,x_3) = x_1 \chi_0(x) + (F_{\lambda}\circ \omega)(x)$ where  $k \geq 1$ is an arbitraty integer and $\lambda(\Omega_1, k,||g_{ij}||_{C^2})$ is sufficiently large. Then the H\"{o}rmander hypo-ellipticity condition is satisfied in $\Omega_1$, that is to say:
\[ D^2\tilde{\phi}_0 (X,X) + D^2\tilde{\phi}_0(\nabla \tilde{\phi}_0,\nabla\tilde{\phi}_0) \geq 0\]
whenever $|X|=|\nabla \tilde{\phi}_0|$ and $\langle \nabla\tilde{\phi}_0, X \rangle =0$.

\end{lemma}

%\begin{lemma}
%Let $ \psi(x_1,x_2,x_3) = x_1 \chi_0(x) + \chi_1(x) e^{\lambda x_3^2}$ where $\chi_0(x)=(1-x_3^{8k})^k$ for $|x_3|\leq 1$ and $\chi_0(x)=0$ otherwise ,$\chi_1(x)=x_3^{2k}$, where  $k \geq 1$ is arbitraty and $\lambda(k, \Omega, k,||g_{ij}||_{C^2})$ is sufficiently large. Then the Hormander hypo-ellipticity condition is satisfied in $\Omega$, that is to say:
%\[ D^2\psi (X,X) + D^2\psi(\nabla \psi,\nabla \psi) \geq 0\]
%whenever $|X|=|\nabla \psi|$ and $\langle \nabla \psi, X \rangle =0$.
%\end{lemma}

\begin{proof}
The proof will be divided into three parts. We will consider the the three regions $A_1=\{ -t_1 \leq x_3 \leq t_2 \}$ , $ A_2=\{  t_2 \leq  x_3 \leq t_2+\delta_2 \} \cup \{-t_1-\delta_1 \leq x_3 \leq -t_1\}$ and $A_3 = \Omega_1 \setminus (A_1 \cup A_2)$ and prove the inequality holds in all these regions. Recall that the  metric is Euclidean on $U_1$ which implies that both $A_1$ and $A_2$ are Euclidean. Let us first consider $A_1$. Note that in this region $ \tilde{\phi}_0(x_1,x_2,x_3) = x_1$ and since the metric is Euclidean in this region we deduce that $ D^2  \tilde{\phi}_0 (X,Y) \equiv 0$ for all $X,Y$ and hence the H\"{o}rmander condition is satisfied.\\

\noindent Let us now focus on the region denoted by $A_3$. Notice that in this region we have $ \tilde{\phi}_0 = F_{\lambda}(\omega(x))$. Therefore the level sets of $\tilde{\phi}_0(x)$ will simply be the level sets $ \{\omega(x) =c\}$. 

$$  D^2\tilde{\phi}_0(X,X) = \langle D_X \nabla \tilde{\phi}_0,X \rangle$$

\noindent Since $ |X| =|F'_{\lambda}(\omega)||\nabla \omega|$ we obtain the following estimate: 
$$  D^2\tilde{\phi}_0(X,X) \leq C |F_{\lambda}'(\omega)|^3$$

\noindent where it is important to note that the constant $C$ is independent of $\lambda$. Furthermore we have:

$$ D^2\tilde{\phi}_0(\nabla \tilde{\phi}_0,\nabla \tilde{\phi_0}) = \frac{1}{2} \nabla \tilde{\phi}_0 (|\nabla \tilde{\phi}_0|^2). $$
Since $\tilde{\phi}_0 = F_\lambda ( \omega(x))$:
$$  D^2\tilde{\phi}_0(\nabla \tilde{\phi}_0,\nabla \tilde{\phi_0})= \frac{1}{2}( F'(\omega)^3 \nabla \omega( |\nabla \omega|^2) + 2F'(\omega)^2F''(\omega) |\nabla \omega|^4).  $$

\noindent One can easily check that for $x \in A_3$:

 \[
    |F'_{\lambda} (\omega)| \leq \left\{\begin{array}{lr}
        C \lambda e^{ \lambda (\frac{\omega-t_2}{\delta_2})^2}& \text{for } t_2+\delta_2 \leq x_3 \\
        C \lambda e^{  \lambda (\frac{\omega+t_1}{\delta_1})^2}, & \text{for } x_3 \leq -t_1-\delta_1\\
        \end{array}\right\}\\
\\
  \]

 \[
    F''_{\lambda} (\omega) \geq \left\{\begin{array}{lr}
        C \lambda^2 e^{ \lambda (\frac{\omega-t_2}{\delta_2})^2}& \text{for } t_2+\delta_2 \leq x_3 \\
        C \lambda^2 e^{  \lambda (\frac{\omega+t_1}{\delta_1})^2}, & \text{for } x_3 \leq -t_1-\delta_1\\
        \end{array}\right\}\\
\\
  \]

\noindent Thus we can easily conclude that for $\lambda$ large enough the H\"{o}rmander hypoellipticity condition is satisfied in this region. Let us now turn our attention to the  transition region $x \in A_2$. Recall that the metric $g$ is flat in $A_2$. We will actually prove the stronger claims:
\\
(1) $ D^2\tilde{\phi}_0(\nabla \tilde{\phi}_0,\nabla \tilde{\phi_0}) \geq 0, $
\\
(2) $ D^2\tilde{\phi}_0 (X,X) \geq 0 $ for all $X$ with $ \langle \nabla \tilde{\phi}_0, X \rangle =0.$
\\
 The idea is that near the $\{x_3=0\}$ hypersurface the convexity of  $x_3^{2k} e^{\lambda x_3^2}$ yields the H\"{o}rmander Hypo Ellipticity. Furthermore away from this surface a suitable choice of $\lambda$ large enough will yield non-negativity as well thus completing the proof. We will now make these statements more precise as follows:

 \[
    F'_{\lambda} (x) = \left\{\begin{array}{lr}
     (\frac{x_3-t_2}{\delta_2})^{2k-1}e^{\lambda(\frac{x_3-t_2}{\delta_2})^2}(\frac{2k+2\lambda(\frac{x_3-t_2}{\delta_2})^2}{\delta_2}) , & \text{for } t_2 \leq x_3 \leq t_2+\delta_2 \\
      (\frac{x_3+t_1}{\delta_1})^{2k-1}e^{\lambda(\frac{x_3+t_1}{\delta_1})^2}(\frac{2k+2\lambda(\frac{x_3+t_1}{\delta_1})^2}{\delta_1}) , & \text{for } -t_1-\delta_1 \leq x_3 \leq -t_1 \\
        \end{array}\right\}\\
\\
  \]

 %\noindent Similarly:

 $$ F''_{\lambda} (x) =  (\frac{x_3-t_2}{\delta_2})^{2k-2} e^{ \lambda(\frac{x_3-t_2}{\delta_2})^2}((\frac{(2k)(2k-1)}{\delta_2^2}+\frac{8\lambda k +2\lambda}{\delta_2^2}(\frac{x_3-t_2}{\delta_2})^2+ \frac{4\lambda^2}{\delta_2^2}(\frac{x_3-t_2}{\delta_2})^4 )$$ for  $t_2 \leq x_3 \leq t_2+\delta_2$ and: \\
 
$$ F''_{\lambda} (x) =  (\frac{x_3+t_1}{\delta_1})^{2k-2} e^{\lambda (\frac{x_3+t_1}{\delta_1})^2}((\frac{(2k)(2k-1)}{\delta_1^2}+\frac{8\lambda k +2\lambda}{\delta_1^2}(\frac{x_3+t_1}{\delta_1})^2+ \frac{4\lambda^2}{\delta_1^2}(\frac{x_3+t_1}{\delta_1})^4 )$$ for  $-t_1-\delta_1 \leq x_3 \leq -t_1$. \\
%\noindent We deduce the following estimate in region $C$ for $\lambda$ large enough:
 %\[
  % | F'_{\lambda} (x)| \leq \left\{\begin{array}{lr}
    %   C \lambda e^{\lambda (\frac{x_3-t_2}{\delta_2})^2}  & \text{for } t_2 \leq x_3 \leq t_2+\delta_2 \\
    %  C \lambda e^{\lambda (\frac{x_3+t_1}{\delta_1})^2}  , & \text{for } -t_1-\delta_1 \leq x_3 \leq -t_1\\
     %   \end{array}\right\}\\
%\\  
%\]
%\noindent Similarly we obtain:

 %\[
   % F''_{\lambda} (x) \geq \left\{\begin{array}{lr}
     %  C \lambda^2 e^{\lambda (\frac{x_3-t_2}{\delta_2})^2}  & \text{for } t_2 \leq x_3 \leq t_2+\delta_2 \\
     % C \lambda^2 e^{\lambda (\frac{x_3+t_1}{\delta_1})^2}  , & \text{for } -t_1-\delta_1 \leq x_3 \leq -t_1\\
      %  \end{array}\right\}\\
%\\  
%\]

\noindent Note that:

$$  D^2\tilde{\phi}_0(\nabla \tilde{\phi}_0,\nabla \tilde{\phi}_0) = (\partial_3 \tilde{\phi}_0)^2 \partial_{33} \tilde{\phi}_0 + 2\partial_1 \tilde{\phi}_0 \partial_3 \tilde{\phi}_0 \partial_{13}\tilde{\phi}_0.$$
So:
$$D^2\tilde{\phi}_0(\nabla \tilde{\phi}_0,\nabla \tilde{\phi}_0) \geq |\partial_3 \tilde{\phi}_0|(|\partial_3\tilde{\phi}_0| \partial_{33}\tilde{\phi}_0 - 2|\chi_0 \chi_0'|).$$
$$  |\partial_3\tilde{\phi}_0| = |x_1 \chi_0' +F'(x_3)| \geq |F'(x_3)| - |x_1| |\chi_0'|. $$
Using the Cauchy-Schwarz inequality we see that:
 \[
    |F'_{\lambda} (x)| \geq \left\{\begin{array}{lr}
        \frac{4}{\delta_2}\sqrt{ \lambda k}(\frac{x_3-t_2}{\delta_2})^{2k} & \text{for } t_2 \leq x_3 \leq t_2+\delta_2\\
     \frac{4}{\delta_2}\sqrt{ \lambda k}(\frac{x_3+t_1}{\delta_1})^{2k} & \text{for } -t_1-\delta_1 \leq x_3 \leq -t_1\\
        \end{array}\right\}\\
\\  
\]
And:
 \[
   |x_1| |\chi_0'| \leq \left\{\begin{array}{lr}
         C(\Omega) k^2  |(\frac{x_3-t_2}{\delta_2})|^{8k-1} & \text{for } t_2 \leq x_3 \leq t_2+\delta_2 \\
      C(\Omega) k^2  |(\frac{x_3+t_1}{\delta_1})|^{8k-1} , & \text{for }  -t_1-\delta_1 \leq x_3 \leq -t_1\\
        \end{array}\right\}\\
\\  
\]
\noindent Hence we can conclude that:
 \[
  |\partial_3\tilde{\phi}_0|  \geq \left\{\begin{array}{lr}
     \frac{4}{\delta_2}\sqrt{ \lambda k}(\frac{x_3-t_2}{\delta_2})^{2k} -  C(\Omega) k^2  |(\frac{x_3-t_2}{\delta_2})|^{8k-1}   & \text{for }  t_2 \leq x_3 \leq t_2+\delta_2 \\
   \frac{4}{\delta_2}\sqrt{ \lambda k}(\frac{x_3+t_1}{\delta_1})^{2k}-  C(\Omega) k^2  |(\frac{x_3+t_1}{\delta_1})|^{8k-1}   , & \text{for }   -t_1-\delta_1 \leq x_3 \leq -t_1\\
        \end{array}\right\}\\
\\  
\]
and therefore for $\lambda$ sufficiently large we obtain that:
 \[
  |\partial_3\tilde{\phi}_0|  \geq \left\{\begin{array}{lr}
     \frac{2}{\delta_2}\sqrt{ \lambda k}(\frac{x_3-t_2}{\delta_2})^{2k}   & \text{for }  t_2 \leq x_3 \leq t_2+\delta_2 \\
   \frac{2}{\delta_2}\sqrt{ \lambda k}(\frac{x_3+t_1}{\delta_1})^{2k}   , & \text{for }   -t_1-\delta_1 \leq x_3 \leq -t_1\\
        \end{array}\right\}\\
\\  
\]

\noindent Now:
$$  \partial_{33}\tilde{\phi}_0= x_1 \chi'_0 + F''(x_3) \geq \frac{1}{2}F''(x_3).$$
Hence:
 \[
   \partial_{33}\tilde{\phi}_0 \geq \left\{\begin{array}{lr}
      \frac{2\lambda}{\delta_2^2} (\frac{x_3-t_2}{\delta_2})^{2k}, & \text{for } t_2 \leq x_3 \leq t_2+\delta_2 \\
    \frac{2\lambda}{\delta_1^2} (\frac{x_3+t_1}{\delta_1})^{2k}, & \text{for }  -t_1-\delta_1 \leq x_3 \leq -t_1\\
        \end{array}\right\}\\
\\  
\]

\noindent Hence combining the above we see that for $\lambda$ sufficiently large we have that:
 $$ D^2\tilde{\phi}_0(\nabla \tilde{\phi}_0,\nabla \tilde{\phi_0}) \geq 0. $$

\noindent Let us now analyze the term  $ D^2\tilde{\phi}_0 (X,X) $ for all $X$ with $ \langle \nabla \tilde{\phi}_0, X \rangle =0$
\\
Note that $d\tilde{\phi}_0(X)=0$ implies that:
$$ X \in \spn \{ \partial_2, \partial_3 \tilde{\phi}_0 \partial_1 - \partial_1 \tilde{\phi}_0 \partial_3 \}.$$
but since $ g$ is Euclidean in this region we have the following:
$$ D^2\tilde{\phi}_0 (\partial_2,X)=0.$$ 
\\
Now:
$$D^2\tilde{\phi}_0 (\partial_3 \tilde{\phi}_0 \partial_1 - \partial_1 \tilde{\phi}_0 \partial_3,\partial_3 \tilde{\phi}_0 \partial_1 - \partial_1 \tilde{\phi}_0 \partial_3) = (\partial_1 \tilde{\phi}_0)^2 \partial_{33} \tilde{\phi}_0- 2\partial_1 \tilde{\phi}_0 \partial_3 \tilde{\phi}_0 \partial_{13}\tilde{\phi}_0.$$
So:
$$D^2\tilde{\phi}_0 (\partial_3 \tilde{\phi}_0 \partial_1 - \partial_1 \tilde{\phi}_0 \partial_3,\partial_3 \tilde{\phi}_0 \partial_1 - \partial_1 \tilde{\phi}_0 \partial_3) =\chi_0^2 (x_1 \chi''_0+ F'') - 2 \chi_0 \chi'_0(x_1 \chi'_0 + F').$$
Using the Cauchy-Schwarz inequality again and by looking at the sign of the $x_3$ we can get the following inequalities:
$$ -2 x_1 |\chi'_0|^2 \chi_0  -2 \chi_0 \chi'_0 F'\geq 0.$$
$$ F'' + x_1 \chi''_0 \geq  \frac{F''}{2} \geq 0.$$

\noindent and thus by combining the above inequalites we obtain that:
$$D^2\tilde{\phi}_0 (\partial_3 \tilde{\phi}_0 \partial_1 - \partial_1 \tilde{\phi}_0 \partial_3,\partial_3 \tilde{\phi}_0 \partial_1 - \partial_1 \tilde{\phi}_0 \partial_3) \geq 0.$$

\end{proof}

We will now provide a suitable modification of the well known fact that the {\bf strict} H\"{o}rmander Hypo-Ellipticity yields a global Carleman estimate.

\begin{lemma}
Let $(\Omega_1,g)$ be a compact smooth Riemannian manifold with smooth boundary and suppose $\psi \in C^{2}(\Omega)$ is such that $d\psi \neq 0$ and the Hormander HypoEllipticity condition is satisfied:
\[ D^2\psi (X,X) + D^2\psi(\nabla \psi,\nabla\psi) \geq 0\]
whenever $|X|=|\nabla\psi|$ and $\langle \nabla\psi, X \rangle =0$.
Then there exists constants $C(\Omega_1,g)$ and $h_0>0$ such that for all $v \in  C^{\infty}_c(\Omega_1)$ and all $ 0<h<h_0$ the following estimate holds:
\[ \| e^{\frac{\psi}{h}} \triangle_g (e^{-\frac{\psi}{h}} v) \|_{L^2(\Omega_1)} \ge \frac{C}{h} \|v\|_{L^2(\Omega_1)} + C \|Dv\|_{L^2(\Omega_1)}  \]
\end{lemma}

\begin{remark}
In general there is a rather standard technique of proving these estimates either through integration by parts or semiclassical calculus. We will employ the former method due to its simplicity. In cases where \[ D^2\psi (X,X) + D^2\psi(\nabla \psi,\nabla\psi) >0 \]
whenever $|X|=|\nabla\psi|$ and $\langle \nabla\psi, X \rangle =0$ one can refer to \cite{EZ} for proving this estimate  where in fact we would get a stronger gain in terms of $h$. Similarly in the case where \[ D^2\psi (X,X) + D^2\psi(\nabla \psi,\nabla\psi) =0 \]
whenever $|X|=|\nabla\psi|$ and $\langle \nabla\psi, X \rangle =0$ one can refer to \cite{DKSU} or \cite{S} for a proof. In our setting we are in an intermediate case and thus require to adjust the arguments. \\
\end{remark}

\begin{proof}

It suffices to prove the claim for the renormalized metric $ \hat{g}=|\nabla^g \psi|_g^2 g $. To see this let us assume that $ c = |\nabla^g \psi|_g^{-2}$ and that $\psi$ is a Carleman weight with respect to $\hat{g}$. But then using the transformation property of 
the Laplace Beltrami operator under conformal changes of metric we deduce that: 
\[ e^{\frac{\psi}{h}} (-h^2 \triangle_g) (e^{-\frac{\psi}{h}} v) =  e^{\frac{\psi}{h}} (-h^2 c^{-\frac{5}{4}} \triangle_{\hat{g}}) (c^{\frac{1}{4}}e^{-\frac{\psi}{h}} v) - h^2 q_c c^{-1} v\]
where:
\[q_c = c^{\frac{1}{4}}\triangle_{c\hat{g}} c^{-\frac{1}{4}}.\]
Now note that $c(x)>0 $ for all $x \in \Omega$ and $ \|q_c\|_{L^{\infty}} < \infty$. Therefore :
\[ \| e^{\frac{\psi}{h}} (-h^2 \triangle_g) (e^{-\frac{\psi}{h}} v)\|_{L^2(g)} \gtrapprox h \|v\|_{L^2} +h^2\|Dv\|_{L^2} - h^2 \|q_c c^{-1}\|_{L^{\infty}} \|v\|_{L^2}.\]
The claim will clearly follow for $h$ small enough.\\
\\

\noindent Let $P_{\psi} :=  e^{\frac{\psi}{h}} (-h^2\triangle_{\hat{g}}) e^{-\frac{\psi}{h}} = A +  B$ where $A$ and $B$ are the formally symmetric and anti-symmetric operators ( in $L^2(\Omega_1, \hat{g})$):\\

\[A= -h^2 \triangle_{\hat{g}} - 1,\] 
\[B= h (2 \langle d\psi,d\cdot\rangle_{\hat{g}} + \triangle_{\hat{g}} \psi).\]\\

\noindent Hence:
\[\|P_{\psi}v\|^2_{L^2(\hat{g})}= \|Av\|^2_{L^2(\hat{g})}+\|Bv\|^2_{L^2(\hat{g})}+ ([A,B]v,v)_{L^2(\hat{g})}.\]\\
Now note that:
\[ [A,B] = -2h^3 [\triangle_{\hat{g}}, \langle d\psi,d\cdot\rangle_{\hat{g}}] + h^3 X.\] 
where $X$ is a smooth vector field.\\

\noindent Let us define the coordinate system $(t,y_1,y_2)$ as follows:
Define the normal vector field to the level sets of $\psi$ and let the integral curves correspond to the coordinate $t$ choosing $t=0$ on one of these level sets. Furthermore let us consider smooth maps $G_t$ to be smooth diffeomorphisms from the unit disk to the corresponding level set $\psi_t$ smoothly depending on $t$. Note that in our coordinate system the pull back of the metric takes the following form : $$g= dt\otimes dt+ g_{\alpha \beta}(t,y) dy^{\alpha}\otimes dy^{\beta}$$

\noindent Thus:
$$([A,B]v,v)_{L^2(\hat{g})} = -2h^3 \int \partial_t \hat{g}^{\alpha\beta} \partial_{\alpha} v \partial_{\beta} v + h^3 \int K(x) |v|^2  $$

\noindent Here, $K$ denotes a continuous function on $\Omega_1$. We now note that  $-\partial_t \hat{g}^{\alpha\beta}$ denotes the inverse of the second fundamental form of the level sets of $\psi$ with respect to the renormalized metric. Recall that if  $ \Gamma^{n-1} \subset M^{n}$ is an embedded nondegenerate hypersurface in $M$, then the second funamental form $h(X,Y)$ on $\Gamma$ changes under conformal rescalings $\hat{g}=cg$ as follows:
\[   \hat{h}(X,X) = \sqrt{c}(h(X,X) + \frac{1}{2}\frac{\nabla_N c}{c} g(X,X)).\] 
Hence:
\[ \hat{h}(X,X)=\sqrt{c}(D^2\psi(X,X) + D^2\psi(\nabla \psi,\nabla \psi) \frac{|X|^2}{|\nabla\psi|^2}). \]
Thus using the main assumption of the Lemma, we see that  $-\partial_t \hat{g}^{\alpha\beta}$ is positive semi-definite and thus we can conclude that:
\[\|P_{\psi}v\|^2_{L^2(\hat{g})} \geq \|Av\|^2_{L^2(\hat{g})}+ \|Bv\|^2_{L^2(\hat{g})}+ ([A,B]v,v)_{L^2(\hat{g})}. \]
So:
\[\|P_{\psi}v\|^2_{L^2(\hat{g})}\geq \|Av\|^2_{L^2(\hat{g})}+  \|Bv\|^2_{L^2(\hat{g})} + h^3 \int K(x) |v|^2  \hspace{2mm} (*)\]
Note that:
\[Bv= h (2 \langle d\psi,dv\rangle_{\hat{g}} + (\triangle_{\hat{g}} \psi)v) = h ( 2\partial_t v + (\triangle_{\hat{g}} \psi) v).\]
The Poincare inequality implies that:
\[  \| \partial_t v\|_{L^2(\Omega_1,\hat{g})} \geq C \| v\|_{L^2(\Omega_1,\hat{g})} \hspace{1cm}  \forall v \in H^1_0 (\Omega_1)\]\\
Recall that the level sets of $\psi$ are non-trapping since $d\psi \neq 0$ anywhere. Since we are working over a compact manifold we can use an integrating factor and use the Poincare inequality above to conclude that:
\[ \|Bv\|_{L^2(\Omega_1,\hat{g})} \geq C h \| v\|_{L^2(\Omega_1,\hat{g})} \hspace{1cm}  \forall v \in C^{\infty}_c(\Omega_1)\hspace{1cm} (**) \]
Let us also observe that by integrating $Av$ against $\delta h^2v$ for some small $\delta$ independent of $h$ we obtain the following estimate:
\[ \|Av\|^2_{L^2(\hat{g})} \geq  C \delta ( h^4 \int |\nabla v|^2 - h^2 \int v^2) \hspace{1cm} (***)\] 

\noindent Combining (*),(**) and (***) yields the claim.

\end{proof}

\begin{corollary}
Let $ \tilde{\phi_0}(x_1,x_2,x_3) = x_1 \chi_0(x_3) + (F_{\lambda} \circ \omega)(x)$ as defined in the previous lemma with  $k \geq 1$ arbitrary and $\lambda$ sufficiently large and only depending on the domain $(\Omega_1,g)$ and on $k$. Then $ \tilde{\phi_0}(x_1,x_2,x_3) $ is a Carleman weight in $(\Omega_1,g)$, that is to say there exists constants $h_0>0$ and $C(\Omega_1,g)$ such that the following estimate holds:

\[ \| e^{\frac{ \tilde{\phi_0}}{h}} \triangle_g (e^{-\frac{ \tilde{\phi_0}}{h}} v) \|_{L^2(\Omega_1)} \ge \frac{C}{h} \|v\|_{L^2(\Omega_1)} + C  \|Dv\|_{L^2(\Omega_1)}  \]
$\forall h \le h_0$ and $ v \in C^{\infty}_{c} (\Omega_1)$.\\
\\

\end{corollary}

\section{Complex Geometric Optics}

\noindent In this section, we will utilize the above corollary to construct a family of solutions to the Schr\"{o}dinger equation $(-\triangle_g + q)u=0$. We will call these solutions complex geometric optic solutions (CGOs).  Before starting this construction let us remark that our CGO solutions will depend on fractional powers of the semiclassical symbol $\frac{1}{\tau}$. This is a new idea and the key reason for using this method is that it will unveil a new analytical method for reconstruction of potential that is different to the known geometric inversion methods for Euclidean geometries ( respectively conformally cylinderical manifolds) such as Radon transforms (respectively geodesic ray transforms)\cite{GU} \cite{DKSU}. The arguments in the next section could be slightly simplified if we abandon these fractional powers. In that case in order to conclude the result we would need the local invertibility of Radon transform \cite{H}.  \\

\noindent Let us start with the notion of complex exponential approximate harmonic functions that concentrate (in some sense) on the plane $\Pi=\{x_3=0\}$ (see for example \cite{GU}). Recall that the plane $\Pi$ is a fixed plane taken out of the foliation $\mathbb{A}$. Let $\epsilon(\tau,\beta) := \tau^{-\beta}$ for a fixed $0<\beta<1$. Choose $M > \frac{2}{\beta}$.  We define $\Phi_{\epsilon}: V \to \mathbb{R}$ and $v_{\epsilon}:V \to \mathbb{R}$ as follows:
$$ \Phi_{\epsilon} = \sum_{k=0}^M \Phi_k(x) \epsilon^k$$
$$ v_{\epsilon} = \sum_{k=0}^M v_k(x) \epsilon^k$$
in such a way that:
$$  \langle d\Phi_{\epsilon}, d\Phi_{\epsilon} \rangle_g  = O(\epsilon^M) = O(\tau^{-2})$$ and:
$$ 2  \langle d\Phi_{\epsilon}, dv_{\epsilon} \rangle_g + (\triangle_g \Phi_{\epsilon})v_{\epsilon} = O(\epsilon^M)=O(\tau^{-2})$$
for $-t_1  \leq x_3 \leq t_2$.
\noindent This is done through iterative determination of the coefficients as follows. We first choose $\Phi_0$ and $v_0$:
$$\Phi_0 = z= x_1 + ix_2$$
$$v_0 = h(z)\chi(x_3)$$
where $h(z)$ is an arbitrary holomorphic function and $\chi$ is an arbitrary smooth function of compact support in $-t_1 \leq x_3 \leq t_2$.
We impose the equations governing the terms $\Phi_k$ by requiring that:
$$ 4 \bar{\partial} \Phi_k + \sum_{j=1}^{k-1} \langle d\Phi_{k-j},d\Phi_j \rangle=0 \hspace{1 cm} \forall k\leq M$$ 
Since the metric is Euclidean for $-t_1 \leq x_3 \leq t_2$ we can in fact solve for an exact $\Phi_{\epsilon}$ as follows:
$$ \Phi_{\epsilon} = z + i \epsilon x_3 + \epsilon^2 \frac{\bar{z}}{4}.$$ 

\noindent Let us observe that $\Re(\Phi_{\epsilon}) = (1+\frac{\epsilon^2}{4})x_1$. We will now rewrite the equations for $v_{\epsilon}$ :
$$ 2 \bar{\partial} v_k + \sum_{j=1}^{2} \langle dv_{k-j},d\Phi_j \rangle=0.$$

\noindent Thus for $k \geq 1$:
$$ 2\bar{\partial} v_k + i \partial_3 v_{k-1} + \frac{1}{2} \partial v_{k-2} = 0   \hspace{1 cm} \forall k\leq M$$

\noindent Let us now make a general remark about the form of $v_k$. Note that $$v_0 = h(z)\chi(x_3)$$ and it is easy to see that we can take $v_1 = -\frac{i}{2}h(z)\chi'(x_3) \bar{z}$. In fact it is not hard to see using induction that:
$v_k = (\frac{-i}{2})^k h(z) \chi^{(k)}(x_3) \frac{\bar{z}^k}{k!}$ modulo lower order terms in $\bar{z}$.More precisely it is possible to choose the $v_k$'s such that:

\[ v_k = \sum_0^{\left \lfloor{\frac{k}{2}}\right \rfloor} a_j \bar{z}^ {k-j} h^{(j)}(z) \chi^{(k-2j)}(x_3) \hspace{2cm} \forall k<M\]

\noindent with:

\[ a_0 = (\frac{-i}{2})^k. \]

\begin{definition}
\noindent Let us define $\tilde{\Phi}_{\epsilon}: \Omega_1 \to \mathbb{C}$ through:
$$\tilde{\Phi}_{\epsilon} =\kappa(\tilde{\phi_0}(x) + i \tilde{\omega}(x) +  i \epsilon \omega(x) + \epsilon^2 (\frac{\tilde{\phi}_0(x)-i\tilde{\omega}(x)}{4})$$ 
Here, $\kappa= \frac{1}{1+\frac{\epsilon^2}{4}}$.\\
\end{definition}

\noindent Notice that for $-t_1 \leq x_3 \leq t_2 $ we have that $ \tilde{\Phi}_{\epsilon}= \kappa \Phi_\epsilon$ and that $\Re(\tilde{\Phi}_{\epsilon})= \tilde{\phi_0}$ for all $x \in \Omega_1$. Finally we note that for $-t_1 \leq x_3 \leq t_2$ we have:
$$  \langle d\tilde{\Phi}_{\epsilon}, d\tilde{\Phi}_{\epsilon} \rangle_g  = O(\epsilon^M) = O(\tau^{-2})$$ and:
$$ 2  \langle d\tilde{\Phi}_{\epsilon}, dv_{\epsilon} \rangle_g + (\triangle_g \tilde{\Phi}_{\epsilon})v_{\epsilon} = O(\epsilon^M)=O(\tau^{-2}).$$
\\

\begin{lemma}
\label{construction new}
Let $ f \in L^2(\Omega_1,g)$ and let $q_*$ be the known function in the statement of Theorem ~\ref{euclid}. For all $\tau>0$ sufficiently large, there exists a unique function $r:=H_{\tau}f $ such that $ P_{\tau}r = e^{-\tau \tilde{\phi_0}} (\triangle_g-q_*) (e^{\tau\tilde{\phi_0}} r) = f$ on $\Omega_1$ with $r$ orthogonal (with respect to the $L^2$ inner product) to $\Sigma$. Here $\Sigma= \left\{ {v\in L^2: (\triangle_g -q_*) (e^{\tau\tilde{\phi_0}}v)=0}\right\} $.
Furthermore:
\[   \|H_{\tau}f\|_{L^2(\Omega_1)} \leq C \tau^{-1} \|f\|_{L^2(\Omega_1)}\]\\
where the constant $C$ only depends on $(\Omega_1,g)$ and $\|q_*\|_{L^{\infty}(\Omega_1)}$.
\end{lemma}

\begin{remark}
This is a rather standard proof about deducing surjectivity for some operator $T$ from the knowledge of injectivity and closed range for the adjoint operator $T^*$. We will closely follow the proofs provided in \cite{NS} and \cite{S}  here. \\
\end{remark}

\begin{proof}
Let $P_{\tau}(\cdot) =  e^{-\tau \tilde{\phi_0}} (\triangle_g-q_*) (e^{\tau\tilde{\phi_0}}\cdot)$ and define $\mathbb{D}= P_{\tau}^* C^{\infty}_c(\Omega_1)$ as a subspace of $L^2(\Omega_1)$. Here $P_{\tau}^*$ denotes the adjoint of $P_{\tau}$ with respect to the standard $L^2(\Omega_1)$ inner product. Consider the linear functional $L: \mathbb{D} \to \mathbb{C}$ through:
\[ L(P^*_{\tau} v) = \langle v,f \rangle \hspace{1cm} \forall v \in C^{\infty}_c(\Omega_1).\]
This is well-defined since any element of $\mathbb{D}$ has a unique representation as $P^*_{\tau} v$  with $ v \in C^{\infty}_c(\Omega_1)$ by the Carleman estimate. Also using Cauchy-Schwarz and the Carleman estimate:
\[ |  L(P^*_{\tau} v)| \leq \|v\|_{L^2} \|f\|_{L^2} \leq C \tau^{-1} \|f\|_{L^2} \|P^*_{\tau} v\|.\]
for $\tau$ large enough with $C$ depending only on $\Omega_1$ and independent of the parameter $\beta$. Thus $L$ is a bounded linear operator on $\mathbb{D}$. Extend $L$ by continuity to the closure $\bar{\mathbb{D}}$ and finally extend to all of $L^2(\Omega_1)$ through projection operator $\pi_{\tau}: L^2(\Omega_1) \to \bar{\mathbb{D}}$. Thus we obtain a bounded linear operator $\hat{L}: L^2(\Omega_1) \to \mathbb{C}$ with $ \hat{L} |_{\mathbb{D}} = L$. Furthermore:
\[ \| \hat{L}\|_{L^2(\Omega_1) \to L^2(\Omega_1)} \leq C \tau^{-1} \|f\|_{L^2(\Omega_1)}.\]
Now by the Riesz representation therorem we deduce that there exists a unique $r \in  L^2(\Omega_1)$ such that $\hat{L}(w) = \langle w,r\rangle \hspace{5mm}\forall w \in L^2(\Omega_1) $ and such that  $\| r\| \leq C \tau^{-1} \|f\|_{L^2}$. Note that:\\
 \[ \langle v,P_{\tau}r \rangle =   \langle P^*_{\tau}v,r\rangle = \hat{L}(P^*_{\tau}v)= L(P^*_{\tau}v)=\langle v,f \rangle. \]
Hence $ P_{\tau}r=f$ in the weak sense and by construction we have $r$ orthogonal to $\Sigma$.\\
\end{proof}
\begin{lemma}
\label{harmonic}
Suppose $0<\beta<1$ is fixed and $\epsilon= \tau^{-\beta}$. For all $\tau>0$ sufficiently large, there exists a solution $u_{\epsilon}$ of  $(-\triangle_g+q_*) u_{\epsilon}=0$ on $\Omega_1$ of the form $$ u^0_{\epsilon} = e^{\tau \tilde{\Phi}_{\epsilon}} (v_{\epsilon} + r^0_{\epsilon})$$
where $\|r^0_{\epsilon}\|_{L^2(\Omega_1)} \leq \frac{C}{\tau}$.
\end{lemma}
\begin{proof}
Let us first consider solving the equation $$ P_{\tau} r :=  e^{-\tau \tilde{\phi_0}} (\triangle_g -q_*) (e^{\tau\tilde{\phi_0}} r)  =-e^{-\tau (\tilde{\phi_0} - \tilde{\Phi}_{\epsilon}  )   } e^{-\tau \tilde{\Phi}_\epsilon} (\triangle_g -q_*) (e^{\tau\tilde{\Phi}_\epsilon} v_{\epsilon}).$$

$$e^{-\tau \tilde{\Phi}_\epsilon} \triangle_g (e^{\tau \tilde{\Phi}_\epsilon} v_{\epsilon})=  \tau^2 \langle d\tilde{\Phi}_\epsilon,d\tilde{\Phi}_\epsilon \rangle_g v_\epsilon + \tau [ 2\langle d\tilde{\Phi}_\epsilon,dv_\epsilon \rangle_g + (\triangle_g \tilde{\Phi}_\epsilon) v_\epsilon ] + \triangle_g v_\epsilon.$$

\noindent Since $v_{\epsilon}$ is compactly supported in the region $-t_1 \leq x_3 \leq t_2$:
$$  \langle d\tilde{\Phi}_{\epsilon}, d\tilde{\Phi}_{\epsilon} \rangle_g  = O(\epsilon^M) = O(\tau^{-2})$$ and:
$$ 2  \langle d\tilde{\Phi}_{\epsilon}, dv_{\epsilon} \rangle_g + (\triangle_g \tilde{\Phi}_{\epsilon})v_{\epsilon} = O(\epsilon^M)=O(\tau^{-2}).$$

\noindent Hence we can immediately conclude that $ \|  e^{-\tau \tilde{\Phi}_\epsilon} (\triangle_g-q_*) (e^{\tau\tilde{\Phi}_\epsilon} v_{\epsilon}) \|_{L^2(\Omega_1)} \leq C$ for some constant $C$. Let $$r= -H_{\tau}(e^{-\tau (\tilde{\phi_0} - \tilde{\Phi}_{\epsilon}  )   } e^{-\tau \tilde{\Phi}_\epsilon} (\triangle_g-q_*) (e^{\tau\tilde{\Phi}_\epsilon} v_{\epsilon})).$$
 
\noindent Clearly $P_{\tau} r =-e^{-\tau (\tilde{\phi_0} - \tilde{\Phi}_{\epsilon}  )   } e^{-\tau \tilde{\Phi}_\epsilon} (\triangle_g-q_*) (e^{\tau\tilde{\Phi}_\epsilon} v_{\epsilon})$ as desired. \\
\noindent Furthermore, since  $ \|  e^{-\tau \tilde{\Phi}_\epsilon} (\triangle_g-q_*) e^{\tau\tilde{\Phi}_\epsilon} v_{\epsilon} \|_{L^2(\Omega_1)} \leq C$ and $\tilde{\phi}_0-\tilde{\Phi}_\epsilon$ is purely imaginary, we can use Lemma ~\ref{construction new} to conclude that for $\tau$ sufficiently large: $ \|r\|_{L^2(\Omega_1} \leq \frac{C}{\tau}$

\noindent We now choose $ r^{0}_{\epsilon} = e^{\tau (\tilde{\phi_0} - \tilde{\Phi}_{\epsilon})} r$ to conclude the proof.\\
\end{proof}

\begin{lemma}
\label{shrodinger}
Let $q_1 \in L^{\infty} (\Omega_1)$. Suppose $0<\beta<1$ is fixed and $\epsilon= \tau^{-\beta}$. For all $\tau>0$ sufficiently large, there exists a solution $u_{\epsilon}^1$ of $(-\triangle_g + q_1)u_{\epsilon}^1=0$ of the form $ u_{\epsilon}^1 = e^{\tau \tilde{\Phi}_{\epsilon}} (v_{\epsilon} + r_{\epsilon}^1)$
where $\|r_{\epsilon}^1\|_{L^2(\Omega_1)} \leq \frac{C(\beta)}{\tau}$.
\end{lemma}

\begin{proof}
Consider the equation:
 $$-e^{-\tau \tilde{\phi_0}} \triangle_g  (e^{\tau\tilde{\phi_0}} r) + q_1r  =-e^{-\tau (\tilde{\phi_0} - \tilde{\Phi}_{\epsilon}  )   } e^{-\tau \tilde{\Phi}_\epsilon} (-\triangle_g+q_1) (e^{\tau\tilde{\Phi}_\epsilon} v_{\epsilon})=:f.$$
but since $v_{\epsilon}$ is compactly supported in $\{ -t_1 \leq x_3 \leq t_2\}$:
$$e^{-\tau \tilde{\Phi}_\epsilon} \triangle_g (e^{\tau \tilde{\Phi}_\epsilon} v_{\epsilon})=  \tau^2 \langle d\tilde{\Phi}_\epsilon,d\tilde{\Phi}_\epsilon \rangle_g v_\epsilon + \tau [ 2\langle d\tilde{\Phi}_\epsilon,dv_\epsilon \rangle_g + (\triangle_g \tilde{\Phi}_\epsilon) v_\epsilon ] + \triangle_g v_\epsilon.$$
\noindent Since $v_{\epsilon}$ is compactly supported in the region $-t_1 \leq x_3 \leq t_2$:
$$  \langle d\tilde{\Phi}_{\epsilon}, d\tilde{\Phi}_{\epsilon} \rangle_g  = O(\epsilon^M) = O(\tau^{-2})$$ and:
$$ 2  \langle d\tilde{\Phi}_{\epsilon}, dv_{\epsilon} \rangle_g + (\triangle_g \tilde{\Phi}_{\epsilon})v_{\epsilon} = O(\epsilon^M)=O(\tau^{-2}).$$ 
\noindent Hence we can immediately conclude that $ \|  e^{-\tau \tilde{\Phi}_\epsilon} \triangle_g e^{\tau\tilde{\Phi}_\epsilon} v_{\epsilon} -q v_\epsilon\|_{L^2(\Omega_1)} \leq C$ for some constant $C$.

\noindent Motivated by Lemma ~\ref{construction new} we make the ansatz $r = H_{\tau} \tilde{r}$ to obtain:
 $$(- I + (q_1-q_*)H_{\tau}) \tilde{r} =f.$$
But $H_\tau: L^2(\Omega_1) \to L^2(\Omega_1)$ is a contraction mapping for $\tau$ large enough with $\|H_{\tau}\| \leq \frac{C}{\tau}$ and thus for sufficiently large $\tau$ the inverse map $(-I+(q_1-q_*)H_{\tau})^{-1}:L^2(\Omega_1) \to L^2(\Omega_1)$ exists and it is given by the following infinite Neumann series:
$$ (-I + (q_1-q_*) H_{\tau})^{-1} = -\sum_{j=0}^{\infty} (q_1-q_*)(H_{\tau})^{j}.$$

\noindent Hence:

$$ \|(-I + (q_1-q_*) H_{\tau})^{-1}\|_{L^2(\Omega_1) \to L^2(\Omega_1)}\leq C.$$

\noindent So we deduce that if :
$$ r = H_\tau ( -I +(q_1-q_*)H_{\tau})^{-1} f.$$ 
then if we choose $ r_{\epsilon}^1 = e^{\tau (\tilde{\phi_0} - \tilde{\Phi}_{\epsilon})} r$  we have that $ u_{\epsilon}^1 = e^{\tau \tilde{\Phi}_{\epsilon}} (v_{\epsilon} + r_{\epsilon}^1)$ solves $(-\triangle_g + q_1)u_{\epsilon}^1=0$ and furthermore:
$$\|r_{\epsilon}^1\|_{L^2(\Omega_1)} \leq \frac{C}{\tau}.$$
\\
\end{proof}

\begin{definition}
\noindent Let $\tilde{\psi_0}(x)= -x_1 \chi_0(x) + (F_{\lambda}\circ \omega)(x)$. (Here, $\chi_0$ and $F_{\lambda}$ are the same functions as in Lemma ~\ref{est})
\end{definition}

\noindent Note that we have the following estimate as a result of Lemma ~\ref{est}:

\[ \| e^{\frac{ \tilde{\psi}_0}{h}} \triangle_g (e^{-\frac{ \tilde{\psi}_0}{h}} v) \|_{L^2(\Omega_1)} \ge \frac{C}{h} \|v\|_{L^2(\Omega_1)} + C  \|Dv\|_{L^2(\Omega_1)}  \]
$\forall h \le h_0$ and $ v \in C^{\infty}_{c} (\Omega_1)$.\\
\\
\noindent Thus we can state the following Lemma which is a direct parallel to Lemma ~\ref{construction new}:

\begin{lemma}
\label{second kernel}
Let $ f \in L^2(\Omega_1,g)$ and let $q_*$ be the known function in the statement of Theorem ~\ref{euclid}. For all $\tau>0$ sufficiently large, there exists a unique function $r:= L_{\tau}f$ such that $ Q_{\tau}r := e^{-\tau \tilde{\psi_0}} (\triangle_g-q_*) (e^{\tau\tilde{\psi_0}} r) = f$ on $\Omega_1$ with $r$ orthogonal (with respect to the $L^2$ inner product) to $\tilde{\Sigma}$. Here $\tilde{\Sigma}= \left\{ {v\in L^2: \triangle_g (e^{\tau\tilde{\psi_0}}v)=0}\right\} $. Furthermore:
\[   \|L_{\tau}f\|_{L^2(\Omega_1)} \leq C \tau^{-1} \|f\|_{L^2(\Omega_1)}\]\\
where the constant $C$ only depends on $(\Omega_1,g)$.
\end{lemma}

\begin{definition}
\noindent Let us define $\tilde{\Psi}_{\epsilon}: \Omega_1 \to \mathbb{C}$ through:
$$\tilde{\Psi}_{\epsilon} =\kappa(\tilde{\psi_0}(x) - i \tilde{\omega}(x) -  i \epsilon \omega(x) + \epsilon^2 (\frac{\tilde{\psi}_0(x)+i\tilde{\omega}(x)}{4})$$ 
Here, $\kappa= \frac{1}{1+\frac{\epsilon^2}{4}}$.\\
\end{definition}

\noindent Notice that for $-t_1 \leq x_3 \leq t_2 $ we have that $ \tilde{\Psi}_{\epsilon}= -\kappa \Phi_\epsilon$ and that $\Re(\tilde{\Psi}_{\epsilon})= -x_1$ for any $x \in \Omega_1$. Finally we note that for $-t_1 \leq x_3 \leq t_2$ we have:
$$  \langle d\tilde{\Psi}_{\epsilon}, d\tilde{\Psi}_{\epsilon} \rangle_g  = O(\epsilon^M) = O(\tau^{-2})$$ and:
$$ 2  \langle d\tilde{\Psi}_{\epsilon}, dv_{\epsilon} \rangle_g + (\triangle_g \tilde{\Psi}_{\epsilon})v_{\epsilon} = O(\epsilon^M)=O(\tau^{-2}).$$
\\
\noindent Thus we can state the following corollary to Theorem ~\ref{shrodinger}:

\begin{corollary}
Let $q_2 \in L^{\infty} (\Omega_1)$. Suppose $0<\beta<1$ is fixed and $\epsilon= \tau^{-\beta}$. For all $\tau>0$ sufficiently large, there exists a solution $u_{\epsilon}^2$ of $(-\triangle_g + q_2)u_{\epsilon}^2=0$ of the form $ u_{\epsilon}^2 = e^{\tau \tilde{\Psi}_{\epsilon}} (v_{\epsilon} + r_{\epsilon}^2)$
where $\|r_{\epsilon}^2\|_{L^2(\Omega_1)} \leq \frac{C(\beta)}{\tau}$.
\end{corollary}

\section{Proof of Uniqueness}

\begin{proof}[Proof of Theorem ~\ref{euclid}]

We start by assuming that $q_1,q_2$ are such that $\Lambda_{q_1}=\Lambda_{q_2}$. Choose an arbitrary $0<\beta<1$, set $\epsilon= \tau^{-\beta}$ and $M= \frac{2}{\beta}$. We use Green's identity pairing $ u_{\epsilon}^1 = e^{\tau \tilde{\Phi}_{\epsilon}} (v_{\epsilon} + r_{\epsilon}^1)$ with $u_{\epsilon}^2=e^{\tau \tilde{\Psi}_{\epsilon}}(v_{\epsilon} + r_{\epsilon}^2)$. Thus:
$$ 0= \int_{\partial \Omega} u_{\epsilon}^1(\Lambda_{q_2}-\Lambda_{q_1})u_{\epsilon}^2 = \int_{\partial\Omega} u_\epsilon^1 \partial_{\nu} u_{\epsilon}^2  - \int_{\partial\Omega} u_{\epsilon}^2 \partial_{\nu} u_\epsilon^1 = \int_{\Omega} u_{\epsilon}^1 \triangle_g u_{\epsilon}^2 - \int_{\Omega} u_{\epsilon}^2 \triangle_g u_{\epsilon}^1. $$
Hence if we let $q=q_2-q_1$ and invoke the equations $(-\triangle_g + q_i)u_\epsilon^i=0$ for $i\in \{1,2\}$ we have:
$$ 0 = \int_{\Omega} q u_{\epsilon}^1 u_{\epsilon}^2 = \int_{V} q u_{\epsilon}^1 u_{\epsilon}^2  = \int_{V} q(v_\epsilon+r_{\epsilon}^1)(v_\epsilon+r_{\epsilon}^2).$$

%\noindent Let $ J_{\epsilon} = \int_{\Omega} v_{\epsilon} e^{\tau \Phi_{\epsilon}} \triangle_g (e^{-\tau \Phi_{\epsilon}} v_{\epsilon})$. Then:
%$$ I_{\epsilon}+J_{\epsilon} = \int_{\Omega} q v_{\epsilon}(v_\epsilon + r_\epsilon^1)  - \int_{\Omega} r_\epsilon^1 e^{\tau \Phi_\epsilon} \triangle( e^{-\tau\Phi_\epsilon} v_{\epsilon})  $$ 
%Hence:
%$$ I_{\epsilon}+J_{\epsilon} = \int_{\Omega} q v_{\epsilon}^2 + \int_{\Omega} q v_{\epsilon}r_\epsilon^1   - \int_{\Omega} r_\epsilon^1 e^{\tau \Phi_\epsilon} \triangle( e^{-\tau\Phi_\epsilon} v_{\epsilon})  $$
%But notice that:
%$$e^{\tau \Phi_\epsilon} \triangle (e^{-\tau\Phi_\epsilon} v_{\epsilon})= \tau^2 \langle d\Phi_\epsilon,d\Phi_\epsilon \rangle_g v_\epsilon -\tau [ 2\langle d\Phi_\epsilon,dv_\epsilon \rangle_g + (\triangle_g \Phi_\epsilon) v_\epsilon ] + \triangle_g v_\epsilon$$
%Hence:
%$$\| e^{\tau \Phi_\epsilon} \triangle (e^{-\tau\Phi_\epsilon} v_{\epsilon})\|_{L^2(\Omega)} \leq C$$ and therefore using Cauchy-Schwarz we see that:
%$$|\int_{\Omega} r_\epsilon^1 e^{\tau \Phi_\epsilon} \triangle (e^{-\tau\Phi_\epsilon} v_{\epsilon})| \leq \frac{C}{\tau}$$
%Similarly we notice that:
%$$|  \int_{\Omega} q v_{\epsilon}r_\epsilon^1| \leq \frac{C}{\tau}$$

\noindent The last equality above holds since $q_1|_{V^c}\equiv q_2|_{V^c}$. Thus using the Cauchy-Schwarz inequality and the fact that $\|r_{\epsilon}^i\|_{L^2(\Omega_1)} \leq \frac{C}{\tau}$ for $\tau$ large and for $i \in \{1,2\}$ we have:
$$ 0 =  \int_{V} q v_{\epsilon}^2 + O(\frac{1}{\tau}).$$

%$$ I_{\epsilon}+J_{\epsilon} = \int_{\Omega} q v_{\epsilon}^2 + O(\frac{1}{\tau})$$

\noindent Let us recall that $ v_{\epsilon}(x) = \sum_{k=0}^M v_k(x) \epsilon^k$ and in fact we have the following formulas:
\[ v_k = \sum_0^{\left \lfloor{\frac{k}{2}}\right \rfloor} a_j \bar{z}^ {k-j} h^{(j)}(z) \chi^{(k-2j)}(x_3) \hspace{2cm} \forall k<M\]

\noindent with:

\[ a_0 = (\frac{-i}{2})^k. \]

\noindent Note that $\int_{\Omega} q v_{\epsilon}^2$ is a polynomial of degree $2M$ in $\epsilon$ and therefore we can conclude that by taking the limit as $\tau \to \infty$ we can determine the coefficients of this polynomial up to the term with coefficient $\epsilon^{\frac{M}{2}}$.
Taking note of the particular form of $v_k$'s and focusing on the coefficient of $\epsilon^k$ for any $k< \frac{M}{2}$ we claim that the knowledge of $\Lambda_{q_1}=\Lambda_{q_2}$ yields the following integral data on $\Pi$:
$$ 0= \int_{\Pi} (\partial_3^{k} q) \bar{z}^k h(z) \hspace{2cm} 0\leq k \leq \frac{M}{2}.$$
Here $h(z)$ is an arbitrary holomorphic function. We will prove this using induction on $k$. Note that by looking at the coefficient of $\epsilon^{0}$ in $\int_{\Omega} q v_{\epsilon}^2$ we obtain the knowledge of $\int_{\Pi} q h(z)^2$.By a simple perturbation argument and choosing $h(z) \to 1+\eta h(z)$ with $\eta \to 0$ we arrive at the following information on the surface $\Pi$:
$$0=\int_{\Pi} q h(z).$$
Let us assume that the claim holds for all $j\leq k-1$. We will now prove the claim for $j=k$. Indeed note that the coefficient of $\epsilon^{k}$ in $\int_{\Omega} q v_{\epsilon}^2$ is equal to:
$$ \int_{\Omega} q (\sum_{s=0}^{k}v_{k-s} v_s).$$
Let us choose $ \chi(x_3) = t^{\frac{1}{2}} \chi_0(t x_3)$ where $\chi_0$ is a non-negative smooth function with compact support with $\int |\chi_0|^2dx_3 = 1$. 
Then the coefficient of  $\epsilon^{k}$ in $\int_{\Omega} q v_{\epsilon}^2$ is equal to:
$$  \int_{\Omega} q h(z)^2 (\frac{-i}{2})^k \frac{1}{k!} \sum_{j=0}^{k}{k \choose j} \chi^{(j)}(x_3) \chi^{(k-j)}(x_3) \bar{z}^k + R.$$ 
and since $$ \sum_{j=0}^{k}{k \choose j} \chi^{(j)}(x_3) \chi^{(k-j)}(x_3)= (\chi^2)^{(k)} (x_3)$$ by taking the limit $t \to \infty$ and using the induction assumption we obtain the following data on $\Pi$:
$$ 0=\int_{\Pi} (\partial_3^{k}q) \bar{z}^k h(z)^2$$ for all $k \geq 0$ and all holomorphic functions $h(z)$. By a simple perturbation argument and choosing $h(z) \to 1+\eta h(z)$ with $\eta \to 0$ we arrive at the following information on the surface $\Pi$:
 $$0=\int_{\Pi} (\partial_3^{k}q) \bar{z}^k h(z)$$ for all $0 \leq k \leq \frac{M}{2}$.
Finally by integrating this data in $x_3$ we obtain the knowledge of:
 $$0= \int_{\Pi} q \bar{z}^k h(z)$$ for all $0 \leq k \leq \frac{M}{2}$ .
Since $M$ can be chosen to be as large as we require this yields the following data on $\Pi$:
 $$ 0=\int_{\Pi} q \bar{z}^k h(z)$$ for all $k\geq 0$ .

\noindent As a final step we note that the subalgebra $\mathbb{A}$ generated by $\{ z^k \bar{z}^j \}$ is separable and unital and therefore by Stone-Weistress theorem we deduce that $q|_{\Pi}=0$  

\end{proof}

\section{Reconstruction Algorithm}

\noindent We next indicate how one may be able to make the above uniqueness proof constructive. In other words we would like to sketch out a reconstruction algorithm for $q$ from the knowledge of $\Lambda_q$. One can immediately observe that the key in accomplishing this would be to construct special solutions $u^{1}_{\epsilon}$ to $(-\triangle_g + q)u^{1}_{\epsilon}=0$ as before and show that we can determine the trace of these solutions on $\partial \Omega$ from the knowledge of Dirichlet to Neumann map $\Lambda_q$. Let us make a few remarks about the approach here. We will be closely following the approach in \cite{NII} and \cite{NS} for the boundary determination of special solutions to the Schr\"{o}dinger equation but we have to make some fundamental changes as we are in a geometry where there is no global limiting Carleman weight. This issue appears to be a key limitation in the determination of boundary values of complex exponential solutions to the Schr\"{o}dinger equation from the DN map and thus we have to make adjustments to the existing arguments here. This will also be the key reason on why we have to assume that $q$ is supported in $U$.

\noindent Let us recall that $H_{\tau}$ denotes the solution operator to $$P_{\tau} v =  e^{-\tau \tilde{\phi_0}} (\triangle_g-q_*) (e^{\tau\tilde{\phi_0}}v)=F.$$ \\
\noindent Let $h_{\tau}(x;y)$ denote the kernel of $H_{\tau}$. Thus:
\[ P_{\tau} v = F \iff v(x) = \int_{\Omega_1} h_{\tau}(x;y) F(y) d\mu_y.\] 

\noindent We will state the following theorem which will help us better understand the regularity of the solutions obtained through the above kernel. In particular the following lemma shows that $ H_{\tau}: L^2(\Omega_1) \to H^2(\Omega_1)$ and that $H_{\tau}^{*}: L^2(\Omega_1) \to H^2(\Omega_1)$.
\begin{lemma}
\label{regularity}
Let $ f \in H^{-2}(\Omega_1)$. There exists a unique weak solution $\tilde{r} \in H^2_0(\Omega_1)$ to:
\[ P_{\tau}P_{-\tau} \tilde{r} = f\]
Furthermore $H_{\tau} f = P_{-\tau} \tilde{r}$.\\
\end{lemma}

\begin{proof}
Note that $  P_{\tau}P_{-\tau} \tilde{r} = f$ implies that $(\triangle_g-q_*)  (e^{2\tau \tilde{\phi}_0} (\triangle_g-q_*) (e^{-\tau \tilde{\phi}_0}\tilde{r})) = e^{\tau \tilde{\phi}_0} f = \hat{f}$. Let $ \hat{r} = e^{-\tau \tilde{\phi}_0} \tilde{r}$. Then it suffices to show that there exists a unique $H^2_0(\Omega_1)$ solution to:
$$(\triangle_g-q_*)  (e^{2\tau \tilde{\phi}_0} (\triangle_g-q_*) \hat{r}) = \hat{f}.$$ 
Note that weak solvability implies that:
$$B(\hat{r},v)=\int_{\Omega_1} e^{2\tau \tilde{\phi}_0} ((\triangle_g-q_*) \hat{r}) ((\triangle_g-q_*) v) = \int_{\Omega_1} \hat{f} v$$ for all $v \in H^2_0(\Omega_1)$.
Note that $B$ is a bilinear bounded operator on $H^2_0(\Omega_1) \times H^2_0(\Omega_1)$ and furthermore using the Poincare and Young inequalities we can obtain the coercivity estimate as well. Thus a simple application of the Lax-Milgram lemma yields the unique solvability of $ P_{\tau}P_{-\tau} \tilde{r} = f$ in $H^2_0(\Omega_1)$.
To see that $ H_{\tau} f = P_{-\tau} \tilde{r}$ we note that for any $ w \in  \Sigma= \left\{ {v\in L^2: (\triangle_g-q_*) (e^{\tau\tilde{\phi}_0}v)=0}\right\}$ we have: $$ (P_{-\tau}\tilde{r} , w)_{L^2} = 0 $$
\end{proof}

\begin{remark}
Note that in the previous proof, one must a priori know that $0$ is not a Dirichlet eigenvalue for the operator $(\triangle_g - q_*)(\triangle_g-q_*)$, but this can be taken for granted as we can always choose an arbitrary $q_0$ such that $q_0 \equiv q_*$ on $U^c$ and such that this hypothesis holds for $q_0$.
\end{remark}

%\noindent Let $\tilde{\psi_0}(x)= -x_1 \chi_0(x) + F_{\lambda}(x)$. Note that we have the following estimate as a result of Lemma ~\ref{est}:

%\[ \| e^{\frac{ \tilde{\psi}_0}{h}} \triangle_g (e^{-\frac{ \tilde{\psi}_0}{h}} v) \|_{L^2(\Omega_1)} \ge \frac{C}{h} \|v\|_{L^2(\Omega_1)} + C  \|Dv\|_{L^2(\Omega_1)}  \]
%$\forall h \le h_0$ and $ v \in C^{\infty}_{c} (\Omega_1)$.\\
%\\
%\noindent Thus we can state the following Lemma which is a direct parallel to Lemma ~\ref{construction new}:

%\begin{lemma}
%Let $ f \in L^2(\Omega_1,g)$. There exists a unique function $r \in H^1(\Omega_1)$ such that $ Q_{\tau}r = e^{-\tau \tilde{\psi_0}} \triangle_g e^{\tau\tilde{\psi_0}} r = f$ on $\Omega_1$ with $r$ orthogonal (with respect to the $L^2$ inner product) to $\Sigma= \left\{ {v\in L^2: \triangle_g (e^{\tau\tilde{\psi_0}}v)=0}\right\} $. (We denote the solution operator by $L_{\tau}$ so that $r = L_{\tau} f$)
%Furthermore for $\tau$ large enough, $r$ satisfies the estimate:
%\[   \|r\|_{L^2(\Omega_1)} \leq C \tau^{-1} \|f\|_{L^2(\Omega_1)}\]\\
%where the constant $C$ only depends on $\Omega_1$.
%\end{lemma}

\noindent Let $l_{\tau}(x;y)$ denote the kernel of $L_{\tau}$ (see Lemma ~\ref{second kernel}). Thus:
\[ Q_{\tau} v = F \iff v(x) = \int_{\Omega_1} l_{\tau}(x;y) F(y) d\mu_y.\]

\noindent Let us now define a new right inverse for the operator  $P_{\tau}  =  e^{-\tau \tilde{\phi_0}} (\triangle_g-q_*) (e^{\tau\tilde{\phi_0}} \cdot )$ through the formula $$ K_\tau = H_{\tau} + \pi_{\tau} L_{\tau}^{*}$$\\
\noindent Here $\pi_{\tau}:L^2(\Omega_1) \to L^2(\Omega_1)$ denotes the orthogonal projection operator of $L^2(\Omega_1)$ functions onto the set: $$\Sigma= \left\{ {v\in L^2: (\triangle_g-q_*) (e^{\tau\tilde{\phi_0}}v)=0}\right\} .$$ One can indeed show that $$\pi_{\tau} : H^{k}(\Omega_1) \to H^{k}(\Omega_1)$$ For a more in depth analysis of this operator we refer the reader to \cite{BU} and \cite{NS}. 
\begin{lemma}
\label{smoothness}
\[ K_{\tau} : L^2(\Omega_1) \to H^2(\Omega_1) \]
\[\|K_{\tau}\|_{L^2(\Omega_1) \to L^2(\Omega_1)} \leq \frac{C}{\tau}.\]
\end{lemma}

\begin{proof}
This is a direct consequence of Lemma ~\ref {regularity}. Note that:
$$K_{\tau} = H_{\tau} + \pi_{\tau} L_{\tau}^{*}.$$
Recall that $ L_{\tau}: H^{-2}(\Omega_1) \to L^2(\Omega_1)$. By duality this implies that $L_{\tau}^{*}: L^2(\Omega_1) \to H^2(\Omega_1)$. Furthermore elliptic regularity implies that $H_{\tau}: L^2(\Omega_1) \to H^2(\Omega_1)$ and $\pi_{\tau}: H^2(\Omega_1) \to H^2(\Omega_1)$. Also note that $\|\pi_{\tau}\|_{L^2 \to L^2} =1$, $\|L_{\tau}^*\|_{L^2 \to L^2}=\|L_{\tau}\|_{L^2 \to L^2} \leq \frac{C}{\tau}$. The norm estimate follows immediately.
\end{proof}
\bigskip
\noindent We will now state a lemma that will be key in accomplishing the boundary determination of special solutions to Schr\"{o}dinger equation. %Improvements to this lemma will probably lead to stronger results.
\\
\begin{lemma}
\label{symmetry}
\[K_{\tau}P_{\tau} v = v\] 
for all $v \in C^{\infty}_c(V)$.
\end{lemma}

\begin{proof}
\[ \int_{V} k_{\tau}(x,y) P_{\tau} v(y) d\mu(y) =  \int_{V} (P_{-\tau} k_{\tau}(x,y))  v(y) d\mu(y) .  \]  
Now recall that $K_{\tau} = H_{\tau}+\pi_{\tau}L^*_{\tau}$. Also note that $P_{-\tau}=Q_{\tau}$ for $y \in V$. Hence:
\[P_{-\tau} K_{\tau}(x,y) = P_{-\tau} h_{\tau}(x,y) + \pi_{\tau}(x,y).\] 
Thus:
\[ \int_{V} g_{\tau}(x,y) P_{\tau} v(y) d\mu(y)= \int_{V} h_{\tau}(x,y) P_{\tau} v(y) d\mu(y)+\int_{V} \pi_{\tau}(x,y) v(y) d\mu(y).\]
Let us note that for any $w \in L^2(\Omega_1)$ we have $w= (1-\pi_{\tau})w + \pi_{\tau} w = P_{-\tau}\tilde{w} + \pi_\tau w$ for some $\tilde{w} \in H^2_0(\Omega_1)$
Hence:
\[ \int_{V} h_{\tau}(x,y) P_{\tau} v(y) d\mu(y) =  \int_{\Omega_1} h_{\tau}(x,y) P_{\tau} v(y) d\mu(y)= \int_{\Omega_1} h_{\tau}(x,y) P_{\tau} P_{-\tau} \tilde{v}(y) d\mu(y) = P_{-\tau} \tilde{v}(x). \] 
Also:
\[ \int_{V} \pi_{\tau}(x,y) v(y) d\mu(y) =  \int_{\Omega_1} \pi_{\tau}(x,y) v(y) d\mu(y) =\pi_{\tau} v.\]
Hence:
\[ \int_{V} k_{\tau}(x,y) P_{\tau} v(y) d\mu(y)= (P_{-\tau} \tilde{v})(x)+ (\pi_{\tau} v)(x) = v(x).\]

\end{proof}

\noindent Let us now define the operator $T_q : L^2(\Omega) \to H^1(\Omega_1)$ through:
\[ T_q (F)(x) := \int_{\Omega} k_{\tau}(x;y)(q(y)-q_*(y))  F(y) d\mu_y.\]
The Carleman estimate in the previous section yields that $T_q:L^2(\Omega) \to L^2(\Omega_1)$ is a contraction mapping with $\|T_q\| \leq \frac{C}{\tau}$ for $\tau$ large enough.Let $a_0 \in H^1(\Omega_1)$ be a solution to $P_{\tau} a_0=0$.  Let us consider the integral equation
 $$a_1(x) = a_0(x) + \int_{\Omega_1}  k_{\tau}(x;y) (q(y)-q_*(y)) a_1(y) d\mu_y.$$
For sufficiently large $\tau$, this integral equation has a unique solution $a_1 \in H^2(\Omega_1)$. Indeed the integral equation is equivalent to:
$$a_1 = a_0+ K_{\tau}((q-q_*)a_1). $$ 
but since $\|K_{\tau}\|_{L^2(\Omega_1) \to L^2(\Omega_1)} \leq \frac{C}{\tau}$ we see that $(I-K_\tau (q-q_*)):L^2(\Omega_1) \to L^2(\Omega_1)$ is invertible with an explicit inverse in terms of a Neumann series and therefore $a_1= (I-K_{\tau}(q-q_*))^{-1}a_0$. Furthermore it is clear from Lemma ~\ref{smoothness} that $a_1 \in H^2(\Omega_1)$.\\

 \noindent Let $G_{\tau}(x;y) = e^{\tau \tilde{\phi_0}(x)} k_{\tau}(x;y) e^{-\tau\tilde{\phi_0}(y)}$ and $ u_i=e^{\tau\tilde{\phi_0}}a_i$ for $i \in \{0,1\}$ \\
%Note that if $ x \neq y$  then $\triangle_g G_{\tau} (x;y) = 0 $
\[ u_1(x) = u_0(x) + \int_{\Omega}  G_{\tau}(x;y) (q(y)-q_*(y)) u_1(y) d\mu_y.\]\\
Hence since $ (-\triangle_g +q)u_1=0$ we have:
\[ u_1(x) = u_0(x) + \int_{V}  G_{\tau}(x;y) (\triangle_g-q_*) u_1(y) d\mu_y .\]\\
Note that $y \in V$ where $V$ denotes the support of $q$. So for $x \in \Omega_1\setminus \bar\Omega$ Green's identity implies that:
\[ u_1(x) = u_0(x) + \int_{\partial V}  G_{\tau}(x;y)\partial_{\nu} u_1(y) d\mu_y- \int_{\partial V} \partial_{\nu}G_{\tau}(x;y) u_1(y) d\mu_y +\int_{V}  (\triangle_g-q_*) G_{\tau}(x;y)) u_1(y) d\mu_y. \]
Recall that Lemma ~\ref{symmetry} implies that for $ x \neq y$ and $y \in V$ we have $(\triangle_g-q_*) G_{\tau} (x;y) = 0 $ (differentiation with respect to the $y$ variable). 
Hence:
\[ u_1(x) = u_0(x) + \int_{\partial V}  G_{\tau}(x;y)\Lambda_q u_1(y) d\mu_y- \int_{\partial V}\partial_{\nu} G_{\tau}(x;y) u_1(y) d\mu_y. \]
\\
For any $f \in H^{\frac{1}{2}}(\partial \Omega)$ let $R_q f \in H^1(\Omega)$ denote the solution operator to $(-\triangle_g + q)u=0$ with $u|_{\partial\Omega}=f$. This solution exists and is unique since $0$ is not a Dirichlet eigenvalue of $-\triangle_g+q.$ \\
Take $x \in \Omega_1\setminus \bar\Omega$ and let $f:=u_1|_{\partial\Omega}$. Then:
\[ \int_{\partial V}  G_{\tau}(x;y)\Lambda_q u_1(y) d\mu_y- \int_{\partial V}\partial_{\nu} G_{\tau}(x;y) u_1(y) = e^{\tau\tilde{\phi_0} (x)} T_q [e^{-\tau\tilde{\phi_0}} R_q(f)](x)=u_1(x) - u_0(x).\]
Define $\Gamma_{\tau}:H^{\frac{1}{2} }(\partial\Omega) \to  H^{\frac{1}{2} }(\partial\Omega)$ through:
\[ \Gamma_{\tau}f=Tr \circ [e^{\tau\tilde{\phi_0}} T_q[e^{-\tau\tilde{\phi_0}} R_q(f)]] .  \]

\noindent We will now state a unique continuation lemma that is a key step in determining the boundary values of special solutions.
\begin{lemma}
\label{unique continuation}
For any $f \in  H^{\frac{1}{2}}(\partial \Omega)$, $\Gamma_{\tau} f$ is known from the knowledge of $\Lambda_q$.
\end{lemma}

\begin{proof}
It suffices to show that for any $f \in  H^{\frac{1}{2}}(\partial \Omega)$ we can determine  $\int_{\partial V}  G_{\tau}(x;y)(\Lambda_q f)(y) d\mu_y- \int_{\partial V}\partial_{\nu} G_{\tau}(x;y) f(y)$ from the knowledge of $\Lambda_q$. We need to define some notation. Note that $\Pi$ divides the manifold into two submanifolds. We will call the one intersecting the set $\{x_3>0\}$ to be $W_u$ and the one intersecting the set $\{x_3<0\}$ to be $W_l$. Let us consider the manifold $W_u \setminus V$. Note that $(\triangle_g-q_*) (R_q f) = 0$ in $W_u \setminus V$. Furthermore we know the Dirichlet and Neumann data for $R_q f$ on $\Omega$ from $\Gamma_q$. We will now proceed to show that this "exterior" pde has a unique solution and thus conclude that $R_q f|_{\partial V}$ and $\partial_{\nu} R_q f|_{\partial V}$ can be determined from $\Lambda_q$. 
Indeed consider the pde $(\triangle_g-q_*) u=0$ in $W_u \setminus V$ with $u|_{\partial \Omega \cap W_u}=f|_{\partial \Omega \cap W_u}$ and $\partial_{\nu} u|_{\partial \Omega \cap W_u}=\Lambda_q f|_{\partial \Omega \cap W_u}$. Suppose there are two solutions $u_1,u_2$ to this pde and consider $v=u_1-u_2$. Since the metric $g$ is Euclidean in a small neighborhood of $V$ we can use Cauchy–Kowalevski theorem to conclude that $v \equiv 0$ in a neighborhood of $V$. We can then use a unique continuation theorem for a second order elliptic operator to conclude that $v \equiv 0$ in $ W_u \setminus V$ ( See for instance \cite{U}).  

\end{proof}

\noindent Observe that we have obtained the following boundary integral equation:
\[ {\bf(I - \Gamma_\tau)(u_1(x)|_{\partial\Omega})= u_0(x)|_{\partial\Omega}}.\]
Thus in order to determine the boundary values of our special solutions to the Schr\"{o}dinger equation it suffices to show that we can uniquely solve the above boundary integral equation. This
will be accomplished through the following theorem. \\

\begin{lemma} 

$\Gamma_{\tau}:  H^{\frac{1}{2} }(\partial\Omega) \to  H^{\frac{1}{2} }(\partial\Omega)$ is compact. Furthermore, for $\tau$ sufficiently large, $\mathbb{N} ( I - \Gamma_{\tau}) = \emptyset$.\\

\end{lemma}

\begin{proof}
\[ \Gamma_{\tau}= Tr\circ[e^{\tau\tilde{\phi_0}(x)} T_q [e^{-\tau\tilde{\phi_0}} R_q(f)]]\]\\

Note that $R_q: H^{\frac{1}{2} }(\partial\Omega) \to H^{1}(\Omega)$ is a bounded linear operator.Secondly, $H^1 \subset\subset L^2$ and $T_q$ is a bounded operator from $L^2$ to $H^1$. Furthermore $Tr: H^1(\Omega_1 \setminus \Omega) \to H^{\frac{1}{2} }(\partial\Omega) $
is bounded. Hence $\Gamma_{\tau}$ is compact.\\
\\
\noindent To prove that the kernel is empty, let us suppose that $( I - \Gamma_{\tau}) f = 0$
Then: $\Gamma_{\tau} f = f$. Let $\rho=T_q(e^{-\tau\tilde{\phi_0}} R_q(f))$. Then:
\[ \rho(x) =  \int_{\Omega} k_{\tau}(x;y)e^{-\tau\tilde{\phi_0}(y)}(q(y)-q_*(y))  R_q(f)(y) d\mu_y.\]
Then:
\[ e^{\tau\tilde{\phi_0}(x)}\rho(x) =  \int_{\Omega} G_{\tau}(x;y)(q(y)-q_*(y))  R_q(f)(y) d\mu_y.\]
Hence:
\[ (\triangle_g-q_*) (e^{\tau\tilde{\phi_0}} \rho) = (q-q_*) R_q(f)= (\triangle_g-q_*)(R_q(f)).\]
But by hypothesis we have that $\Gamma_\tau(f) = f$ so $e^{\tau\tilde{\phi_0}} \rho|_{\partial\Omega} = f $  Hence $e^{\tau\tilde{\phi_0}(x)} \rho(x)= R_{q}(f)(x)$ on $\Omega$
so:
\[ \rho= T_q(\rho).\]
Finally since $T_q$ is a contraction mapping for $\tau$ large enough, we deduce that $ \rho = 0$ everywhere. Hence $f=0$.
\bigskip
\end{proof}

\section{Further Results}

In this section we will state a few theorems which generalize Theorem ~\ref{euclid}. For the sake of brevity we will only indicate the key differences of the proofs to that of Theorem ~\ref{euclid}.  As a first step in generalizing Theorem ~\ref{euclid} recall that we had assumed $\Gamma$ to be connected. This assumption can be relaxed significantly. Indeed one can strengthen the result by allowing $\Gamma$ to have several components. In this case a more delicate Carleman weight needs to be constructed. This will be the main content of Theorem ~\ref{euclid2}.\\

%\noindent In Theorem ~\ref{euclid} we implicitly assume that $q$ is supported in the Euclidean region. Naturally one would like to treat the most general case where the potential $q$ is unknown everywhere in the manifold. This is at the moment beyond the scope of this paper as we do not have a reconstruction algorithm for the trace of CGO solutions. However, we can still make the result stronger by assuming that the potential $q$ is explicitly known in $U^c$. This will indeed be the content of Theorem ~\ref{support}.\\ 

\noindent Finally we will seek to generalize the result further by assuming that the metric restricted to the subdomain $U$ is conformally transversally anisotropic. These are manifolds $(U,g)$ where $U = \mathbb{R} \times U_0$ and:
\[ g =c(x)(dx_1^2 + g_0(x')) \] 
Here, $g_0$ denotes the induced metric on the transversal submanifold $U_0$ which is independent of $x_1$. Indeed recall that a key step in proving Theorem ~\ref{euclid} is the construction of a global phase function which is locally a limiting Carleman weight in $U$. In [5] it is shown that local existence of limiting Carleman weights restricts the geometry to CTA geometries. Hence it would seem natural to expect Theorem ~\ref{euclid} to have a generalization in this setting. This will be the content of Theorem ~\ref{cta}.\\ 

\begin{theorem}
\label{euclid2}
Let $(\Omega^3,g)$ denote a compact smooth Riemannian manifold with smooth boundary. Let $U \subset \Omega$ be an open subset such that $\Gamma=U \cap \partial \Omega$ is non-empty, and strictly convex. Let $\Gamma=\cup_{i=1}^{l} \Gamma^i$ where $\Gamma^i$ denotes the connected components of $\Gamma$. Let us assume that $U$ can be covered with coordinate charts in which $g|_U$ is the Euclidean metric. Let $U^i$ denote the convex hull of $\Gamma_i$ and let us assume that $U^i \cap U^j = \emptyset \hspace{2mm} \forall i,j$ and that $U = \cup_{i=1}^l U^i$. Suppose $q$ is a smooth function and that $q-q_*$ is compactly supported in $U$ where $q_*$ is a globally known smooth function. Then the knowledge of $\Lambda_q$ will uniquely determine $q$.\\ 
\end{theorem}

\noindent Indeed one can see that the key in establishing this theorem is proving a similar Carleman estimate to that of Lemma ~\ref{est}. The rest of the techniques in the paper including the CGO solutions with fractional powers of $\frac{1}{\tau}$ and the reconstruction algorithms would be exactly as before. Let us now give a sketch of how one can prove a Carleman estimate in this setting.\\

\noindent First, Let us extend the manifold $\Omega$ to a slightly larger manifold $\Omega_1$. We extend $q$ to all of $\Omega_1$ by setting it equal to zero outside $\Omega$ and extend $g$ smoothly to $\Omega_1$ such that $ g|_{U_1 }$ is Euclidean. Here $U_1$ denotes the extension of convex hull of $\Gamma$ to the larger manifold $\Omega_1$. Note that $U_1=\cup_1^l \mathbb{A}_i$ where $\mathbb{A}_i$ is a foliation by a family of planes $\mathbb{A}_i=\{\Pi^i_t\}_{t \in I}.$ We start by taking a fixed family of planes $\Pi^{i} \in \mathbb{A}_i$ for all $1\leq i\leq l$ . A local coordinate system $(x^i_1,x^i_2,x^i_3)$ can be constructed in each $U_1^i$ such that $\Pi^i=\{x^i_3=0\}$ with $(x^i_1,x^i_2)$ denoting the usual Euclidean coordinate system on the plane $\Pi^i$ and $x^i_3$ denoting the normal flow to this plane. We can assume that within each component $U^i$ the support of $q$ lies in the compact set $V^i \subset\subset \{-t^i_1<x^i_3<t^i_2\}$ with $t^i_1,t^i_2>0$. In this framework $U_1^i= \cup_{c=-t^i_1-\delta^i_1}^{c=t^i_2+\delta^i_2} \{x^i_3=c\}$ with $\delta>0$. Let $\omega:\Omega_1 \to \mathbb{R}$ be any smooth function such that $d\omega \neq 0$ everywhere and $\omega(x) \equiv x^i_3$  for $ x \in U_1^i$ for each $ 1\leq i\leq l$. The existence of such $\omega$ is proved in Lemma ~\ref{global}. Let us define two globally defined $C^{k}(\overline{\Omega_1})$ functions $\chi_0 : \Omega_1 \to \mathbb{R}$ and $F_{\lambda}(x): \mathbb{R} \to \mathbb{R}$  as follows:
 \[
    \chi_0(x) = \left\{\begin{array}{lr}
        1, & \text{for }  -t^i_1<x^i_3<t^i_2\\
        (1-(\frac{x^i_3-t^i_2}{\delta^i_2})^{8k})^k, & \text{for } t^i_2 \leq x^i_3 \leq t^i_2+\delta^i_2\\
    (1-(\frac{x^i_3+t^i_1}{\delta^i_1})^{8k})^k, & \text{for } -t^i_1-\delta^i_1 \leq x^i_3 \leq -t^i_1\\
         0  & \text{otherwise }
        \end{array}\right\}
  \]

 \[
    F_{\lambda} (x) = \left\{\begin{array}{lr}
        0, & \text{for }   -t^i_1<x<t^i_2\\
        e^{\lambda (\frac{x-t^i_2}{\delta_2})^2} (\frac{x-t^i_2}{\delta^i_2})^{2k}, & \text{for } t^i_2 \leq x  \\
       e^{\lambda (\frac{x+t^i_1}{\delta_1})^2} (\frac{x+t^i_1}{\delta^i_1})^{2k}  , & \text{for }  x \leq -t^i_1\\
        \end{array}\right\}\\
\\
  \]

\noindent Motivated by Lemma ~\ref{est} we will construct a global Carleman weight in $\Omega_1$ in such a way that the phase function restricted to the support of $q$ is equivalent to a limiting Carleman weight. In the transition regions we will use a convexification technique in order to make sure that the H\"{o}rmander hypo-ellipticity condition holds.
\\

\begin{lemma}
\label{est2}

Let $ \tilde{\phi}_0(x_1,x_2,x_3) = x^i_1 \chi_0(x) + (F_{\lambda}\circ \omega)(x)$ where  $k \geq 1$ is an arbitraty integer and $\lambda(\Omega_1, k,||g_{ij}||_{C^2})$ is sufficiently large. Then the H\"{o}rmander hypo-ellipticity condition is satisfied in $\Omega_1$, that is to say:
\[ D^2\tilde{\phi}_0 (X,X) + D^2\tilde{\phi}_0(\nabla \tilde{\phi}_0,\nabla\tilde{\phi}_0) \geq 0\]
whenever $|X|=|\nabla \tilde{\phi}_0|$ and $\langle \nabla\tilde{\phi}_0, X \rangle =0$.

\end{lemma}

\noindent Indeed a detailed look at the proof of Lemma ~\ref{est} suggests that the key to proving Lemma ~\ref{est2} would be the existence of the function $\omega$:

\begin{lemma}
\label{global}
There exists a smooth function $\omega:\Omega_1 \to \mathbb{R}$ such that $d\omega \neq 0$ everywhere and $\omega(x) \equiv x^i_3$  for $ x \in U_1^i$ for each $ 1\leq i\leq l$.
\end{lemma}

\noindent Before proving Lemma ~\ref{global} let us recall Morse Lemma. This will be the key ingredient of the proof. Heuristically the idea is to start with an arbitrary smooth $\omega_0(x) \equiv x^i_3$  for $ x \in U_1^i$ for each $ 1\leq i\leq l$ and then pull out all the critical points to reach at the desired function $\omega$.

\begin{lemma}[Morse Lemma]
Let $b$ be a non-degenerate critical point of $f: \Omega_1\to \mathbb{R}$. Then there exists a chart $(x_1,x_2,x_3)$ in a neighborhood of $b$ such that $$f(x)= f(b)-x_1^2-x_2^2-...-x_{\alpha}^2 + x_{\alpha+1}^2+...+x_n^2$$ Here $\alpha$ is equal to the index of $f$ at $b$.
\end{lemma}

\begin{proof}[Proof of Lemma ~\ref{global}]
 Define $\omega_0: \Omega_1 \to \mathbb{R}$ such that $\omega_0(x) \equiv x^i_3$  for $ x \in U_1^i$ for each $ 1\leq i\leq l$. Let us remind the reader that the $(x^i_1,x^i_2,x^i_3)$ are essentially the locally well defined Fermi coordinates near the planes $\Pi^i$.
We know that a Generic smooth function is Morse and therefore it has isolated critical points. Thus by using a small $C^{\infty}$ purturbation we can find a smooth function $\omega_1(x)$ such that $\omega_1(x) \equiv x^i_3$  for $ x \in U_1^i$ for each $ 1\leq i\leq l$ and $\omega_1(x)$ has isolated critical points and thus by compactness a finite number of isolated critical points $b_k$ for $1 \leq k \leq L$. We will assume without loss of geneality that the index of these critical points is zero.\\
Since $\dim \Omega_1=3>2$, we can connect these critical points with points just outside the boundary by a family of disjoint paths that do not intersect $U$. We will denote these curves by $\gamma_k$.\\
Let $V_k$ denote the neighborhood around $b_k$ for which the Morse lemma holds. Choose $h$ small enough such that the geodesic ball of radius $h$ around $b_k$ is inside $V_k$ namely $B_{b_k}(h) \subset V_k$. Take $$\omega_2(x)= \omega_1(x) + \epsilon (\alpha x_1 + \beta x_2 +\lambda x_3) \eta_k(x)$$ where $\eta_k$ is a smooth function compactly supported in $V_k$ and such that $\eta_k \equiv 1 $ in the ball $B_{b_k}(\frac{h}{2})$. It is clear that for $\epsilon$ small enough we still have that $\omega_2(x) \equiv x^i_3$  for $ x \in U_1^i$ for each $ 1\leq i\leq l$ . Furthermore we can see that for $\epsilon$ small enough the critical points of $\omega_2$ outside $V_k$ will remain the same and the critical point of $\omega_2$ inside $V_k$ must be in the ball $B_{b_k}(\frac{h}{2})$. Hence the critical point in $V_k$ will 'move' from $b_k$ to the point with local coordinates $(x_1,x_2,x_3) = (\frac{\epsilon \alpha}{2},\frac{\epsilon \beta}{2},\frac{\epsilon \lambda}{2})$. Since $\Omega_1$ is compact, it is clear that we can 'move' the critcal points $b_k$ along their respective curves $\gamma_k$ and essentially construct a smooth function $\omega$ with $\omega(x) \equiv x^i_3$  for $ x \in U_1^i$ for each $ 1\leq i\leq l$ and such that $|d\omega|_g \neq 0$ anywhere in $\Omega_1$. \\
\end{proof}

\noindent With the proof of Lemma ~\ref{global} complete we can deduce easily that Lemma ~\ref{est2} must also hold. This in turn implies that we can construct CGO solutions with fractional powers in the semiclassical parameter $\frac{1}{\tau}$ as in the previous section concentrating on the planes $\Pi^i$. One can also use the techniques in the previous sections to obtain the trace of these CGO solutions on the boundary through obtaining a Fredholm type boundary integral equation and thus give a reconstruction. \\
\bigskip

\bigskip

\noindent Let us now discuss the generalization to CTA geometries. In a sense this is the most general statement one could hope for, given the present tools in this paper. Before stating the Theorem, let us explain some notions. We will assume that $\Omega = I \times \Omega_0$ where $I=[a,b]$ is a compact interval. Let $U_0 \subset \Omega_0$ and define $U=I \times U_0$. Suppose that $(U, g|_U)$ is a CTA geometry, that is to say there exists a coordinate chart such that $g|_U =c(x)(dx_1^2 + g_0(x'))$. We have the following:

\begin{theorem}
\label{cta}
Let $(\Omega^3,g)$ denote a compact smooth Riemannian manifold with smooth boundary with $\Omega=I \times \Omega_0$. Let $U_0$ be an open subset such that $\Gamma_0=U_0 \cap \partial \Omega_0$ is non-empty, strictly convex and that $U_0$ is the convex hull of $\Gamma_0$. Suppose that $(\partial I \times U_0) \cup (I \times \partial U_0)$ is connected, and let $U= I \times U_0$. Suppose $q$ is a smooth function that is explicitly known in $U^c$, and that $U$ can be covered with a coordinate chart in which $(U, g|_U)$ is conformally transversally anisotropic. Then the knowledge of $\Lambda_q$ will uniquely determine $q$ in $U$ provided that $U_0$ is simple  and the geodesic transform on $\Omega_0$ is locally injective.\\ 
\end{theorem}

\begin{remark}
\noindent The proof of this theorem is not given in its entirety as there will be several overlaps with the Euclidean case. In particular we will prove the Carleman estimate and construct the CGO solutions. We will also show how the CGO solutions will yield the uniqueness of potential. We will however omit the boundary reconstruction algorithm of the CGO solutions as that will be exactly as in the Euclidean case. 
\end{remark}

\begin{remark}
This proof can in turn be extended as in the Euclidean case to potentials known outside a multiply connected region to provide a generalization of Theorem ~\ref{euclid2}.
\end{remark}

\noindent Let us first observe that the Laplace operator in three dimensions transforms under the following law for conformal rescalings of the metric:
\[ \triangle_{fg} u = f^{-\frac{5}{4}} ( \triangle_g + q_f)(f^{\frac{1}{4}}u)\]
\[q_f = f^{\frac{1}{4}}\triangle_{fg} f^{-\frac{1}{4}}\]

\noindent Using this, we see that without loss of generality we can assume that in $U$ the metric $g$ takes the following form:
\[ g = dx_1^2 + g_0(x')\]
\noindent In other words, without loss of generality we can assume $c(x) \equiv 1$. We will also assume without loss of generality that $q_* \equiv 0$. 
\bigskip

\noindent First, Let us extend the manifold $\Omega_0$ to a slightly larger manifold $\Omega_{01}$ and let $\Omega_1= I \times \Omega_{01}$. We extend $q$ to all of $\Omega_1$ by setting it equal to zero outside $\Omega$ and extend $g$ smoothly to $\Omega_1$ such that $ g|_{U_1 }=dx_1^2+g_0(x')$. Here $U_1=I \times U_{10}$ denotes the extension of $U$ to the larger manifold $\Omega_{1}$. Set $\Gamma_1 = U_1 \cap \partial \Omega_{01}$ and note that $U_{10}$ is the convex hull of $\Gamma_1$. $\partial U_{10}$ consists of 4 curves. There will be two segments shared with $\partial \Omega_1$ and two geodesics $\gamma_1$ and $\gamma_2$. One should think of $\gamma_2$ as being above $\gamma_1$ according to the orientation of the manifold. Let us choose $\tilde{\gamma_2}$ to be a strictly convex curve just below $\gamma_2$ and take $\tilde{\gamma}_1$ to be a strictly concave curve just above $\gamma_1$. These curves will both exists as they can just be taken to be curves with sufficiently small constant mean curvatures. Existence of such curves will then be immediate as they are governed by second order ordinary differential equations. Let this region be denoted by $V_0$ and let $V:=I \times V_0$. We note that the support of the potential $q$ lies in the set $V$. We will construct a phase function that will essentially be a limiting Carleman weight in this region and then transitions to a globally well defined smooth function in such a way that the H\"{o}rmander hypoellipticity condition is satisfied. We will denote the region outside of $V$ and above $I \times \tilde{\gamma}_2$ by $W_u$ and the other remaining region outside of $V$ and below $\tilde{\gamma}_1$ by $W_l$.
\noindent Let us first construct the local Fermi coordinates $(x_1,x_2,x_3)$ about the surface $I \times \tilde{\gamma}_2$ and $(x_1,y_2,y_3)$ about $I \times \tilde{\gamma}_1$. 

\noindent Note that near $I \times \tilde{\gamma}_2$ we have :
\[ g = dx_1^2 + dx_3^2 + \rho(x_2,x_3) dx_2^2 \] 

\noindent and near $I \times \tilde{\gamma}_1$:
\[ g = dx_1^2 + dy_3^2 + \tilde{\rho}(y_2,y_3) dy_2^2\]

\noindent The convexity conditions imply that locally near the two surfaces we have:
$$ \partial_3 \rho <0$$
and:
$$ \partial_3 \tilde{\rho} >0$$

\noindent Let $\delta_1,\delta_2$ be sufficiently small and define $\omega:\Omega_1 \to \mathbb{R}$ to be any smooth function such that $d\omega \neq 0$ everywhere,  $\omega(x) \equiv x_3$  for $ x \in W_u \cap \{0 \leq x_3 \leq \delta_2\}$ and $\omega(x) \equiv y_3$  for $ x \in W_l \cap \{-\delta_1 \leq y_3 \leq 0\}$. \\ 
Define the two functions $\chi_0$ and $F_{\lambda}$ as follows:
 \[
    \chi_0(x) = \left\{\begin{array}{lr}
        1, & \text{for }  x \in V\\
        (1-(\frac{x_3}{\delta_2})^{8k})^k, & \text{for } 0 \leq x_3\leq\delta_2\\
    (1-(\frac{y_3}{\delta_1})^{8k})^k, & \text{for } -\delta_1 \leq y_3 \leq 0\\
         0  & \text{otherwise }
        \end{array}\right\}
  \]

 \[
    F_{\lambda} (x) = \left\{\begin{array}{lr}
        0, & \text{for }   x \in V\\
        e^{\lambda (\frac{\omega(x)}{\delta_2})^2} (\frac{\omega(x)}{\delta_2})^{2k}, & \text{for } x \in W_u \\
       e^{\lambda (\frac{\omega(x)}{\delta_1})^2} (\frac{\omega(x)}{\delta_1})^{2k}  , & \text{for }  x\in W_l\\
        \end{array}\right\}\\
\\
  \]

%Let $\alpha <0$ be sufficiently large and define:
%$$ \tilde{x} = e^{\alpha x}$$
%$$ \tilde{t}_1 = e^{\alpha t_1}$$
%$$\tilde{t}_2=e^{-\alpha t_2}$$
%$$\tilde{\delta_2}=e^{\alpha(t_2+\delta_2)}-e^{\alpha t_2}$$
%$$\tilde{\delta_2}=-e^{\alpha(-t_1-\delta_1)}+e^{-\alpha t_1}$$

%\noindent Let us define two globally defined $C^{k}(\overline{\Omega_1})$ functions $\chi_0 : \Omega_1 \to \mathbb{R}$ and $F_{\lambda}(x): \mathbb{R} \to \mathbb{R}$  as follows:
 %\[
   % \chi_0(x) = \left\{\begin{array}{lr}
     %   1, & \text{for }  -t_1<x_3<t_2\\
       % (1-(\frac{\tilde{x}_3-\tilde{t}_2}{\tilde{\delta}_2})^{8k})^k, & \text{for } t_2 \leq x_3 \leq t_2+\delta_2\\
    %(1-(\frac{\tilde{x}_3-\tilde{t}_1}{\tilde{\delta}_1})^{8k})^k, & \text{for } -t_1-\delta_1 \leq x_3 \leq -t_1\\
       %  0  & \text{otherwise }
        %\end{array}\right\}
 % \]

 %\[
 %   F_{\lambda} (x) = \left\{\begin{array}{lr}
    %    0, & \text{for }   -t_1<x<t_2\\
      % e^{\lambda (\frac{\tilde{x}-\tilde{t}_2}{\tilde{\delta}_2})^2} (\frac{\tilde{x}-\tilde{t}_2}{\tilde{\delta}_2})^{2k}, & \text{for } t_2 \leq x  \\
     %  e^{\lambda (\frac{\tilde{x}-\tilde{t}_1}{\tilde{\delta}_1})^2} (\frac{\tilde{x}-\tilde{t_1}}{\tilde{\delta}_1})^{2k}  , & \text{for }  x \leq -t_1\\
       % \end{array}\right\}\\
%\\
  %\]

\begin{lemma}
\label{est cta}

Let $ \tilde{\phi}_0(x_1,x_2,x_3) = x_1 \chi_0(x) + F_{\lambda}(x)$ where  $k \geq 1$ is an arbitraty integer and $\lambda(\Omega_1, k,||g_{ij}||_{C^2})$ is sufficiently large. Then the H\"{o}rmander hypo-ellipticity condition is satisfied in $\Omega_1$, that is to say:
\[ D^2\tilde{\phi}_0 (X,X) + D^2\tilde{\phi}_0(\nabla \tilde{\phi}_0,\nabla\tilde{\phi}_0) \geq 0\]
whenever $|X|=|\nabla \tilde{\phi}_0|$ and $\langle \nabla\tilde{\phi}_0, X \rangle =0$.\\
\end{lemma}

\begin{proof}[Proof of Lemma ~\ref{est cta}]
 We will consider the the five regions $V$ , $ A=\{  0 \leq  x_3 \leq \delta_2 \}$, $B =\{-\delta_1\leq y_3\leq 0\}$, $W_u \setminus A$ and $W_l \setminus B$ and prove the inequality holds in all these regions. Regions $B$ and $W_l \setminus B$ can be handled in the exact same manner as regions $A$ and $W_u \setminus A$ respectively so we only focus on the three regions $V$, $A$ and $W_u \setminus A$.
Let us first consider $V$. Note that in this region $ \tilde{\phi}_0(x_1,x_2,x_3) = x_1$ and since the metric is conformally transversally anisotropic in this region we can use the result in \cite{DKSU} to deduce that the H\"{o}rmander condition is satisfied in $A$. This is due to the fact that $x_1$ is a limiting Carleman weight in this region. \\

\noindent The region denoted by $W_u\setminus A$. This region can be handled in the exact same manner as in Lemma ~\ref{est} due to the convexification method implemented in the design of the function $F_{\lambda}$. We will now focus on the region denoted by $A$;

\noindent Recall that in Riemannian geometries:
\[ Hess(f) := D^2 f = (\partial_j \partial_k f - \Gamma_{jk}^l\partial_l f) dx^j \otimes dx^k\]
\noindent Where $\Gamma_{jk}^l$ denotes the Christoffel symbol defined through:
\[ \Gamma_{jk}^l = \frac{1}{2} g^{lm}(\partial_k g_{mj} +\partial_j g_{mk}-\partial_m g_{jk})\]

\noindent Note that in the transition region $A$ the only non-zero terms in Christoffel symbol are $\Gamma^2_{22},\Gamma^2_{23}, \Gamma^3_{22}$. Also note that $\nabla \tilde{\phi}_0 = \partial_1 \tilde{\phi}_0 \partial_1 + \partial_3 \tilde{\phi}_0 \partial_3$. Hence:

$$  D^2\tilde{\phi}_0(\nabla \tilde{\phi}_0,\nabla \tilde{\phi}_0) = (\partial_3 \tilde{\phi}_0)^2 \partial_{33} \tilde{\phi}_0 + 2 \partial_1 \tilde{\phi}_0 \partial_3 \tilde{\phi}_0 \partial_{13}\tilde{\phi}_0 $$

\noindent Indeed one can see that this is exactly the same expression as in the Euclidean setting and thus we have:

$$D^2\tilde{\phi}_0(\nabla \tilde{\phi}_0,\nabla \tilde{\phi}_0) \geq 0$$

\noindent Let us now analyze the term  $ D^2\tilde{\phi}_0 (X,X) $ for all $X$ with $ \langle \nabla \tilde{\phi}_0, X \rangle =0$
\\
Note that $d\tilde{\phi}_0(X)=0$ implies that:
$$ X \in \spn \{ \partial_2,Z \}$$
Where $Z= \partial_3 \tilde{\phi}_0 \partial_1 - \partial_1 \tilde{\phi}_0 \partial_3$. In this region we have the following:
$$ D^2\tilde{\phi}_0 (\partial_2,\partial_1)=0$$
$$ D^2\tilde{\phi}_0 (\partial_2,\partial_2)=- \Gamma_{22}^3 \partial_3 \tilde{\phi}_0$$ 
$$ D^2\tilde{\phi}_0 (\partial_2,\partial_3)= 0  $$

\noindent One should note that the convexity assumption on $\gamma_2$ implies that :
$$ D^2\tilde{\phi}_0 (\partial_2,\partial_2) \geq 0 $$

\noindent Note that:
$$D^2\tilde{\phi}_0 (\partial_3 \tilde{\phi}_0 \partial_1 - \partial_1 \tilde{\phi}_0 \partial_3,\partial_3 \tilde{\phi}_0 \partial_1 - \partial_1 \tilde{\phi}_0 \partial_3) = (\partial_1 \tilde{\phi}_0)^2 \partial_{33} \tilde{\phi}_0- 2\partial_1 \tilde{\phi}_0 \partial_3 \tilde{\phi}_0 \partial_{13}\tilde{\phi}_0-\Gamma_{33}^3 (\partial_1 \tilde{\phi}_0)^2\partial_3\tilde{\phi}_0$$

\noindent Again, the reader can see that this expression is exactly as in the Euclidean setting and thus the exact same argument applies here to coclude that:
$$D^2\tilde{\phi}_0 (\partial_3 \tilde{\phi}_0 \partial_1 - \partial_1 \tilde{\phi}_0 \partial_3,\partial_3 \tilde{\phi}_0 \partial_1 - \partial_1 \tilde{\phi}_0 \partial_3) \geq 0$$

\end{proof}

\noindent With the proof of Lemma ~\ref{est cta} now complete, one can proceed with construction of the CGO solutions as follows. Let $p$ be a point just outside $\Omega_{01}$ such that the geodesic $\gamma$ emanating from $p$ will be in the region between $\tilde{\gamma_1}$ and $\tilde{\gamma_2}$. Let us choose the normal coordinate system $(r,\theta)$ about $p$ so that we have the following:
$$ g = dx_1^2 + dr^2 + c(r,\theta) d\theta^2$$

\noindent Let us define $\Phi = \tilde{\phi}_0 + i r$. We also define $v_0 = c^{\frac{1}{4}} h(x_1+ir) \chi(\theta)$ where $h$ is an arbitrary holomorphic function and $\chi$ is an arbitrary function of compact support near $\gamma$. We have the following two Lemmas which are exact parallels to Lemma ~\ref{harmonic} and Lemma ~\ref{shrodinger}.

\begin{lemma}
\label{harmonic2}
There exists a family of exact solutions $u_0$ to $-\triangle_g u_0=0$ of the form $ u_0 = e^{\tau \Phi} (v_0 + r_0)$
where $\|r_0\|_{L^2(\Omega_1)} \leq \frac{C}{\tau}$.
\end{lemma}

\begin{lemma}
\label{shrodinger2}
Let $q \in L^{\infty} (\Omega_1)$. There exists a family of exact solutions $u_1$ to $(-\triangle_g + q)u_1=0$ of form $ u_1 = e^{\tau \Phi} (v_0 + r_1)$
where $\|r_1\|_{L^2(\Omega_1)} \leq \frac{C}{\tau}$.
\end{lemma}

\begin{remark}
The proofs of Lemmas ~\ref{harmonic2} and \ref{shrodinger2} are omitted as they are similar to Lemmas ~\ref{harmonic} and ~\ref{shrodinger}. The reader should also note that the proof of the reconstruction algorithm provided in the Euclidean setting can be duplicated here to obtain the trace of the CGO solutions provided above on $\partial \Omega$.
\end{remark}

\begin{proof}[Proof of Theorem ~\ref{cta}]

 Let $w = e^{-\tau \Phi} v_0$ and use the Green identity by pairing $ u_1 = e^{\tau \Phi} (v_0 + r_1)$ with $w$. Thus:
$$  I = \int_{\partial\Omega} w \partial_{\nu} u_1  - \int_{\partial\Omega} u_1 \partial_{\nu} w = \int_{\Omega} w \triangle_g u_1 - \int_{\Omega} u_1 \triangle_g w $$
Hence:
$$ I= \int_{\Omega} q w u_1 - \int_{\Omega} e^{\tau \Phi} (v_0 + r_1) \triangle_g  (e^{-\tau \Phi} v_1) $$

\noindent Let $ J= \int_{\Omega} v_0 e^{\tau \Phi} \triangle_g (e^{-\tau \Phi} v_0)$. Then:
$$ I+J = \int_{\Omega} q v_0(v_0 + r_1)  - \int_{\Omega} r_1 e^{\tau \Phi} \triangle( e^{-\tau\Phi} v_0)  $$ 
But notice that:
$$e^{\tau \Phi} \triangle (e^{-\tau\Phi} v_0)= \tau^2 \langle d\Phi,d\Phi \rangle_g v_0 -\tau [ 2\langle d\Phi,dv_0 \rangle_g + (\triangle_g \Phi) v_0 ] + \triangle_g v_0$$
Hence:
$$\| e^{\tau \Phi} \triangle (e^{-\tau\Phi} v_0)\|_{L^2(\Omega)} \leq C$$ and therefore using Cauchy-Schwarz we see that:
$$|\int_{\Omega} r_1 e^{\tau \Phi} \triangle (e^{-\tau\Phi} v_0)| \leq \frac{C}{\tau}$$
Similarly we notice that:
$$|  \int_{\Omega} q v_0r_1| \leq \frac{C}{\tau}$$

\noindent Thus:

$$ I+J = \int_{\Omega} q v_0^2 + O(\frac{1}{\tau})$$ 

\noindent Thus similarly to the Euclidean setting we see that by choosing $\chi(\theta)$ approximating a delta distribution we can obtain the knowledge of following data:
$$\int_{I \times \gamma} q h(z) $$

\noindent By choosing $h(z) = e^{-i \lambda z}$ we see that we have the knowledge of the following data:
$$ \int_{\gamma} \hat{q}(\lambda,r,\theta) e^{\lambda r} dr $$
\noindent where $\hat{q}$ denotes the fourier transform of $q$ in $x_1$. 

\noindent Let us note that by assumption the local geodesic transform on $U$ is invertible. We will use this to conclude that $q$ can be determined from the data above. Indeed suppose that we have a function $f$ such that :
$$ \int_{\gamma} f(\lambda,r,\theta) e^{\lambda r} dr = 0 \hspace{2cm} (*)$$

\noindent Setting $\lambda=0$ and using the injectivity of the geodesic ray transform we can conclude that $f(0,r,\theta)=0$. Now differentiating (*) with respect to $\lambda$ and setting $\lambda=0$ yields that $\partial_{\lambda} f |_{\lambda=0} \equiv 0$. Repeating this argument yields that all derivatives of $f$ must be zero. Since $f$ represents the Fourier transform of a compactly supported function it must be real analytic and therefore it must vanish everywhere. Thus we can conclude that the integral data $\int_{I \times \gamma} q h(z) $ indeed reconstructs the potential $q$ uniquely in $V$.\\
\end{proof}

\begin{remark}
Using the Gaussian beam quasi mode construction in \cite{DKLS} the result in Theorem ~\ref{cta} can be strengthened by removing the simplicity restriction required on $U_0$ and instead imposing the injectivity of geodesic ray transform on $U_0$. Examples of such manifolds for which this injectivity is known, are simple manifolds, manifolds which are Foliations by strictly convex hypersurfaces \cite{DKLS}[25]. \\
\end{remark}

% Every LaTeX document must end with \end{document}.

\end{document}